\documentclass[12pt]{amsart}%

\usepackage{amsmath, amssymb, amsthm, amsfonts, bbm, dsfont}
\usepackage{graphicx}
\usepackage{wrapfig}
\usepackage{mathrsfs}

\usepackage{hyperref}

\usepackage{a4wide}

\numberwithin{equation}{section}
\theoremstyle{plain}
\newtheorem{theorem}[subsubsection]{Theorem}
 \newtheorem{lemma}[subsubsection]{Lemma}
 \newtheorem{proposition}[subsubsection]{Proposition}
 \newtheorem{corollary}[subsubsection]{Corollary}
 \newtheorem{conjecture}[subsubsection]{Conjecture}

 \theoremstyle{definition}

\usepackage{pictexwd}
\usepackage{stackengine}
\usepackage{mathtools}
\usepackage{amssymb}
\usepackage{dsfont}

\usepackage{mathabx}

\newcommand{\CC}{\mathbb{C}}
\newcommand{\cc}{\mathbb{C}}

\newcommand{\LL}{\mathbb{L}}

\newcommand{\zz}{\mathbb{Z}}

\newcommand{\scX}{\mathscr{X}}

\newcommand{\calE}{\mathcal{E}}
\newcommand{\calF}{\mathcal{F}}
\newcommand{\calc}{\mathcal{C}}

\newcommand{\calG}{\mathcal{G}}

\newcommand{\calO}{\mathcal{O}}

\newcommand{\calB}{\mathcal{B}}

\newcommand{\cpp}{\mathcal{P}}
\newcommand{\chh}{\mathcal{H}}

\newcommand{\frg}{\mathfrak{g}}
\newcommand{\frp}{\mathfrak{p}}

\newcommand{\frb}{\mathfrak{b}}

\newcommand{\frn}{\mathfrak{n}}

\newcommand{\calcr}{\bar{\calc}}

\newcommand\FT{\dot{\mathfrak{F}}}
\newcommand{\FTi}{\mathfrak{F}}
\newcommand{\FTT}{\mathfrak{F}}

\newcommand{\ind}{\textup{ind}}

\newcommand\Spec{\textup{Spec}}

\newcommand{\Tr}{\textup{Tr}}

\newcommand\GL{\textup{GL}}

\newcommand{\Ad}{\textup{Ad}}

\newcommand{\eps}{\epsilon}

\newcommand{\quash}[1]{}

\newcommand{\Ltimes}{\stackrel{\LL}{\otimes}}

\newcommand{\und}[1]{\underline{#1}}
\newcommand{\ui}{\underline{i}}

\newcommand{\Br}{\mathfrak{Br}}
\newcommand{\LBr}{\mathfrak{LBr}}
\newcommand{\PBr}{\mathfrak{PBr}}

\usepackage{tikz}
\usetikzlibrary{shapes,fit,intersections}
\usetikzlibrary{decorations.markings}
\usetikzlibrary{decorations.pathreplacing}
\usetikzlibrary{patterns}
\usetikzlibrary{matrix,arrows}
\usepgflibrary{decorations.pathmorphing}

\usepackage{tikz-cd}

\newcommand{\Wr}{\overline{W}}
\newcommand{\Hilb}{\textup{Hilb}}

\newcommand{\MF}{\mathrm{MF}}
\newcommand{\calXr}{\overline{\mathcal{X}}}
\newcommand{\calX}{\mathcal{X}}

\newcommand{\calZ}{\mathcal{Z}}
\newcommand{\calC}{\mathcal{C}}

\newcommand{\scXr}{\overline{\mathscr{X}}}
\newcommand{\uscX}{\underline{\mathscr{X}}}
\newcommand{\uscXr}{\underline{\scXr}}

\newcommand{\ucx}{\underline{\mathscr{X}}}

\newcommand{\frh}{\mathfrak{h}}

\newcommand{\CE}{\mathrm{CE}}

\newcommand{\HH}{\textup{H}}

\newcommand{\dia}{\mathop{\Delta}}

\newcommand{\forg}{\mathrm{fgt}}

\newcommand{\Tqt}{\mathbb{T}_{q,t}}

\newcommand{\KN}{\mathrm{KN}}

\newcommand{\HY}{\mathrm{HXY}}

\newcommand{\dic}{\mathrm{dic}}
\newcommand{\lbr}{\mathrm{lbr}}
\newcommand{\br}{\mathrm{br}}

\newcommand{\Brn}{\Br_n}

\def\Hilb{ \mathrm{Hilb}}

\newcommand{\uMF}{\underline{\mathrm{MF}}}
\newcommand{\uoMF}{\overline{\uMF}}

\newcommand{\ti}{\times}
\newcommand{\oti}{\otimes}
\newcommand{\ot}{\otimes}
\newcommand{\stry}{\stackengine{1pt}{$\bar{\star}$}{$\scriptstyle y$}{U}{c}{F}{T}{S}}
\newcommand{\strx}{\stackengine{1pt}{$\bar{\star}$}{$\scriptstyle x$}{U}{c}{F}{T}{S}}
\newcommand{\stty}{\stackengine{1pt}{$\tilde{\star}$}{$\scriptstyle y$}{U}{c}{F}{T}{S}}
\newcommand{\sttx}{\stackengine{1pt}{$\tilde{\star}$}{$\scriptstyle x$}{U}{c}{F}{T}{S}}

\title{Dualizable link homology}
\author{A. Oblomkov}
\address{
A.~Oblomkov\\
Department of Mathematics and Statistics\\
University of Massachusetts at Amherst\\
Lederle Graduate Research Tower\\
710 N. Pleasant Street\\
Amherst, MA 01003 USA
}
\email{oblomkov@math.umass.edu}

\author{L. Rozansky}
\address{
L.~Rozansky\\
Department of Mathematics\\
University of North Carolina at Chapel Hill\\
CB \# 3250, Phillips Hall\\
Chapel Hill, NC 27599 USA
}
\email{rozansky@math.unc.edu}

\begin{document}
\maketitle

\def\fF{ \mathfrak{F}}
\def\HMP{HOMFLY-PT}
\def\HMPp{\HMP\ polynomial}
\def\spl{super-polynomial}
\def\Yfd{Y-fied}
\def\Yfcn{Y-fication}
\def\bfx{\mathbf{x}}
\def\bfy{\mathbf{y}}
\def\IC{\CC}
\def\ICbxy{\IC[\bfx,\bfy]}
\def\lc{\ell}
\def\xop{\mathrm{op}}
\def\Brn{\Br_n}
\def\Brnop{ \Brn^{\xop}}
\def\FTi{ I }
\def\mfL{ \mathfrak{L}}
\def\babeta{\bar{\beta}}
\def\baL{\bar{L}}
\def\Lv#1{ L_{#1}}
\def\Lb{ \Lv{\beta}}

\def\bfxv#1{\bfx_{#1}}
\def\bfyv#1{\bfy_{#1}}
\def\bfxl{ \bfxv{\lc}}
\def\bfyl{ \bfyv{\lc}}
\def\bfxn{ \bfxv{n}}
\def\bfyn{ \bfyv{n}}

\def\Tqt{ \mathbb{T}_{q,t}}
\def\Csv#1{ \CC^*_{#1}}
\def\Csq{\Csv{q}}
\def\Cst{ \Csv{t}}

\def\aR{ R }
\def\aRv#1{ \aR(#1)}
\def\aRL{ \aRv{L}}
\def\aRr{ \aR^{\mathrm{r}} }
\def\aRrv#1{ \aRr(#1)}
\def\aRrL{ \aRrv{L}}

\def\lob{ \mathcal{E}}
\def\lobv#1{ \lob(#1)}
\def\lobL{ \lobv{L}}
\def\lobLb{ \lobv{\Lb}}
\def\lobbaL{ \lobv{\baL}}

\def\pldr{palindromic}

\def\Dper{ \mathrm{D}^{\mathrm{per}} }
\def\DpT{ \Dper_{\Tqt}}
\def\DpTL{ \DpT(\aRrL-\mathrm{mod})}
\def\Cxy{ \mathcal{C} }
\def\Clc{ \Cxy_{\lc}}
\def\MCxy{ M_{\Cxy}}
\def\dcr{dichromatic}
\def\rmH{\mathrm{H}}
\def\HCxy{ \rmH_{\Cxy}}
\def\HClc{ \rmH_{\Clc}}
\def\Tor{ \mathrm{Tor}}

\def\Knr{Kn\"{o}rrer}

\begin{abstract}
We modify our previous construction of link homology in order to include a natural duality functor $\fF$.
  To a link \(L\) we associate a triply-graded module \(\HY(L)\) over the graded polynomial ring
  \(R(L)=\cc[x_1,y_1,\dots,x_\ell,y_\ell]\). The module has an involution \(\mathfrak{F}\) that intertwines
  the Fourier transform on \(R(L)\), \(\mathfrak{F}(x_i)=y_i\), \(\mathfrak{F}(y_i)=x_i\).
  In the case when \(\ell=1\) the module is free over \(R(L)\) and specialization to \(x=y=0\)
  matches with the triply-graded knot homology
  previously constructed by
  the authors. Thus we show that the corresponding super-polynomial  satisfies the categorical version of \(q\to 1/q\) symmetry.
  We also construct an isotopy invariant of the closure of a dichromatic braid and relate this invariant to
  \(\HY(L)\).
  \end{abstract}

\section{Introduction}
\label{sec:introduction}

It is easy to see from the skein relations~\cite{Jones87} 
that the \HMPp\ of a knot \(P_K(a,q)\in \zz[a,q^{\pm 1}]\)
has a symmetry: \(P_K(a,q)=P_K(a,-1/q)\). In this paper we prove the conjecture that this symmetry lifts to the link homology.

Currently there are two triply-graded knot homologies
\cite{KhovanovRozansky08b}, \cite{OblomkovRozansky16} whose doubly-graded Euler characteristic equals the \HMPp.
It is expected that these homologies are equivalent, although this has not been proven yet. In this
paper we study the knot homology of \cite{OblomkovRozansky16}.

The double-graded Poincare polynomial of the link homology 
%
\(\cpp_K(a,q,t)\)  is simply called the super-polynomial.
%
By its definition, its specialization at $t=-1$ is the \HMPp:  \(\cpp_K(a,q,-1)=P_K(a,q)\). We prove that similar to the \HMPp, 
the \spl\ is also palindromic.
\begin{theorem} \label{thm:main}
The \spl\ of any knot $K$ is palindromic:
  \[\cpp_K(a,q,t)=\cpp_K(a,tq^{-1},t).\]
\end{theorem}

In fact we upgrade the knot homology to an object of the derived category with a natural notion of Fourier
transform and show that the object is preserved by the transform. The theorem above follows after we apply
the derived functor of global section. We give an outline of the result below.

For  the triply-graded homology \cite{KhovanovRozansky08b} the statement of the theorem was stated in the
work of Dunfield, Gukov and Rasmussen \cite{DunfieldGukovRasmussen06} as
a conjecture with a lot of numerical support. The equivalence of the homology theories \cite{KhovanovRozansky08b} and
\cite{OblomkovRozansky16} would  imply the original conjecture in \cite{DunfieldGukovRasmussen06}.
\subsection{An object-valued link invariant and its symmetry}

The following construction was motivated  by the 
constructions of  Batson-Seed \cite{BatsonSeed15}, Cautis-Kamnitzer
\cite{CautisKamntizer16} and
Gorsky-Hogancamp \cite{GorskyHogancamp17} and by an observation in~\cite{KhovanovRozansky16} that the link homology of an $\lc$-component link can be promoted to an object in the derived category of modules over $\cc[x_1,\dots,x_\lc]$.
The  homology theory developed in \cite{GorskyHogancamp17} is particularly close to our theory, see the last section
of the paper.

As usual, all our categories and vector spaces are equivariant with respect to the torus $\Tqt = \Csq\times\Cst$.
For a positive integer $n$ we denote $\bfxn =x_1,\ldots,x_n$ and $\bfyn=y_1,\ldots,y_n$. Their $\Tqt$ weight are
\[
\deg x_i = q^2,\qquad \deg y_i = q^{-2} t^2.
\]

In this paper $\lc$ denotes the number of components of a link $L$, and we consider an associated algebra $\aRL=\CC[\bfxl,\bfyl]$.
The derived category of $\Tqt$-equivariant (2-periodic) $\aRL$-modules $\DpT\bigl(R(L)\bigr)$
has a
Fourier endofunctor $\FTT$ which transposes the variables
and changes the generators of $\Tqt$:
\begin{equation}\label{eq:Four-ring}
  \FTT(x_i) = y_i,\quad \FTT(y_i) = -x_i,\quad \FTT(q) = tq^{-1},\quad \FTT(t) = t.
\end{equation}

To a closure $L(\beta)$ of a braid $\beta$ we associate an object $\calE(L(\beta))$ of the category $\DpT(R(L))$.
By modifying argument of \cite{OblomkovRozansky16} we prove
\begin{theorem}
The object $\calE(L(\beta))$ is invariant under the Markov moves, thus representing an invariant $\lobL$ of an oriented link $L$.
\end{theorem}

In contrast to the link homology defined \cite{OblomkovRozansky16}, the invariant \(\calE(L)\) is symmetric with respect to the Fourier involution:
\begin{theorem}\label{thm:obj}
  The object $\lobL$ is invariant with respect to the Fourier involution:
  \[\FTT\bigl(\lobL\bigr)\cong\lobL.\]
\end{theorem}

The  Fourier transform discussed here is related to several other Fourier transforms.
It is conjectured that the homology of the algebraic knot are equal to the
homology of the Hilbert scheme of points on the corresponding singular curve
\cite{OblomkovRasmussenShende12}. In this algebro-geometric context the
Fourier transform manifests itself as Serre duality for the stable pairs
on the curve (see \cite{OblomkovRasmussenShende12}). There is also
conjectural relation between the homology of the torus links and
the representations of the rational Cherednik algebras \cite{GorskyOblomkovRasmussenShende14}.
In this setting the Fourier transform becomes the Fourier transform for modules
over the rational Cherednik algebras.

It is also easy to see that the object $\lobL$
is invariant with respect to the simultaneous reversal of orientation of all components of $L$:
\begin{equation}\label{eq:reverse}
  \lobbaL \cong \lobL,
\end{equation}
where $\baL$ is the link with reversed orientation.

\subsection{Dualizable homology}
\label{sec:dualizable-homology}

If we apply the functor of global sections to our invariant we obtain {\it dualizable homology}:
\[\HY(L(\beta))=R\Gamma(\mathcal{E}(L(\beta))).\]

Thus we see that the extended homology have palindromic
property:

\begin{theorem} \label{thm:symm}
  For any link \(L\) we have:
  \[\dim_{q,t,a} \HY(L)=\dim_{t^2/q,t,a} \HY(L).\]
\end{theorem}

The  relation between the extended homology \(\HY(L)\) and the triply-graded homology \(\mathrm{H}(L)\)
from \cite{OblomkovRozansky16} is subtle and explained in the next subsection but in the case of knots all technicalities evaporate and we are left with

\begin{theorem}\label{thm:free}
  For any knot \(K\) the module \(\HY(K)\) is
  free and finite over \(R(K)=\cc[x,y]\) and
  \[\HY(K)|_{y=0}=\HH(K).\]
\end{theorem}

In particular the \(\cc[x]\)-module \(\HH(K)\) is free and there is a finite dimensional
triply-graded vector space \(\overline{\HH}(K)\) such that:
\begin{equation}\label{eq:fd-hom}
  \HH(K)=\overline{\HH}(K)\otimes \cc[x].\end{equation}

The  graded dimension of the vector space \(\overline{\mathrm{H}}(K)\) is what we call super-polynomial \(\cpp(K)(a,q,t)\)
of the knot \(K\). Thus the theorem~\ref{thm:main} follows immediately from the previous statements.

In future work \cite{OblomkovRozansky18c} we plan to explore a version of the constructions from \cite{OblomkovRozansky18a} for the homology theory
\(\HY\) developed here. In particular, the sheaves  related to the homology  \(\HY\) have less singular support compare to
the sheaves constructed in \cite{OblomkovRozansky18a} thus we expect that it would be possible to find a connection between
the localization formalism for \(\HY\) homology theory and the conjectural theory of projectors proposed in \cite{GorskyNegutRasmussen16}.

\subsection{Invariant of the closure of a dichromatic braid}
\label{sec:invar-clos-dichr}

Generalizing the construction from \cite{OblomkovRozansky16} we obtain a homomorphism from the
groupoid of two-colored braids into a special monoidal category of matrix factorizations:
\[\Phi^{\dic}:\Br_n^{\dic}\to \overline{\MF}[\bullet,\bullet].\]

Using this homomorphism we construct the trace functor from this groupoid:
\begin{equation}\label{eq:dicE}
  \mathcal{E}: \Br_{n;\ell}^{\dic}\to \DpT (\cc^\ell)\end{equation}
where \(\Br_{n;\ell}^{\dic}\) is the set of closable dichromatic braids with \(\ell\) connected
components of the closure. We show that the trace is actually an isotopy invariant of the
closure.

Suppose \(L=L(\und{\beta})\), \(\und{\beta}\in \Br_{n;\ell}\) is a two-colored link.
A subset $\Cxy\subset\Clc = \{1,2,\ldots,\lc\}$ determines a coloring of each component $L_i$ of $L$ by either an `$x$-box' (if $i\notin \Cxy$) or a `$y$-box' (if $i\in\Cxy$). Define a corresponding $R(L(\beta))$-module $R(L(\und{\beta}))$ as a quotient of $R(L(\beta))$
by $\lc$ conditions: for each $i$ set $y_i = 0$ if $i\not\in\Cxy$ and $x_i=0$ if $i\in\Cxy$. The module \(R(L(\und{\beta}))\) is
a ring and \(\Spec(R(\und{\beta}))\) is \(\CC^\ell\) appearing in \eqref{eq:dicE}.

When the braid has only one color, \(\Cxy=\emptyset\),
then the space of derived global sections of the trace
is the old invariant from the paper \cite{OblomkovRozansky16}.
The object $\calE(L(\beta))$ determines all colored objects through the derived restriction:
\begin{theorem} For any \(\und{\beta}\in \Br_n^{\dic}\) we have
\[\mathcal{E}(L(\und{\beta}))=\mathcal{E}(L(\beta))\Ltimes_{R(L(\beta))}R(L(\und{\beta})).\]
\end{theorem}


{\bf Acknowledgments:}  The authors would like to thank you Eugene Gorsky, Andrei Negut and Matt Hogancamp for
useful discussions. Also we are thankful to organizers of the workshop "Hidden Algebraic structures in Topology" in  CalTech in March of 2019 for inviting the authors, the discussion of
\(q\to 1/q\) at the workshop spurred us to think about extending our previous construction.
The work of A.O. was supported in part by  the NSF CAREER grant DMS-1352398, NSF FRG grant DMS-1760373 and Simons Fellowship.
The work of L.R. was supported in part by  the NSF grant DMS-1108727.

\def\yLg{Long}
\def\ylg{long}
\def\ySh{Short}
\def\ysh{short}

\section{Dualizable category}
\label{sec:dual-cat}

\subsection{Notations}

In this text we do not discuss induction and restriction functors, so we can fix the size of our matrices to be \(n\) and use the standard notations \(\mathfrak{g}\), \(\mathfrak{b}\), \(\mathfrak{n}\)  for Lie algebras associated with the group \(G=\GL_n\).
We also use notation \(\dia: \frb\to \cc^n\) for the linear map that extract the diagonal of the matrix.
The  same notation is used for linear map \(\dia: \frb\to \frh\) when we identify \(\CC^n\) with \(\frh\)
in the most natural way.

The other set of notations are borrowed from our previous papers. In particular, \(X_+\) \(X_{++}\) denote
 upper-triangular and strictly upper-triangular parts of the matrix \(X\), so that \(X=X_-+X_{++}=X_{--}+X_+\).
 The symmetry group \(S_n\) acts on  upper-triangular matrices by permuting their diagonal entries:
 \[\sigma(X)_{ij}=X_{ij}-\delta_{i,j}(X_{ii}-X_{\sigma(i),\sigma(i)}).\]
 In particular \((\sigma \cdot X)_{ii}=X_{\sigma(i),\sigma(i)}\)


 \subsection{\yLg\ categories}
\label{sec:big-categories}

Our basic category  is the monodromic version of the category of \cite{OblomkovRozansky16}. The monodromic categories have advantage over the category from \cite{OblomkovRozansky16}: the Fourier transform acts on them as endofunctor. The monodromic category has two equivalent presentations, related by the Kn\"{o}rrer periocidity: the \ylg\ one and the \ysh\ one. The \ylg\ category has an obvious monoidal structure, but the action of the Fourier endofunctor is obscured. The \ysh\ category has an obvious Fourier symmetry, but the convolution looks unnatural and its Fourier symmetry requires a special proof.

The \ylg\ category is the category of \(G\ti B \ti B\)-equivariant matrix factorizations over the space
\[\underline{\scX_2}\subset\frg \ti G\ti \frb\ti G\ti\frb, \quad
  \underline{\scX_2}=\{\dia(Y_1)=\dia(Y_2)\},\]
  \[(g,b_1,b_2)\cdot (X,g_1,Y_1,g_2,Y_2)= (\Ad_g(X),g_1\cdot b_1,\Ad_{b_1}Y_1, g_2\cdot b_2,\Ad_{b_2}Y_2).\]

  For a pair of permutations \(\sigma,\tau\) we define a \(G\ti B^2\)-invariant potential on \(\underline{\scX_2}\):
  \begin{equation}\label{eq:pot1}W_{\sigma,\tau}(X,g_1,Y_1,g_2,Y_2)=\Tr(X(\Ad_{g_1}(\sigma\cdot
    Y_1)-\Ad_{g_2}(\tau\cdot Y_2))).
  \end{equation}
This   potential is invariant with respect to the first factor of
  \(T_{qt}=\cc^*_q\ti \cc^*_t\) and has weight \(2\) with respect to the second factor:
  \[(\lambda,\mu)\cdot (X,g_1,Y_1,g_2,Y_2)=(\lambda X,g_1,\lambda^{-1}\mu^2 Y_1,g_2,\lambda^{-1}\mu^2Y_2).\]

 Thus the \ylg\ category is  the category of \(G\ti B^2\) matrix factorization which have two-periodic differentials of \(\cc^*_q\)
  degree \(1\):
  \[\uMF_{\sigma, \tau}:=\MF_{G\ti B^2}(\underline{\scX_2},W_{\sigma,\tau}).\]

  The convolution space \(\underline{\scX_3}=\frg\ti (G\ti \frb)^3\) has
  three \(G\ti B^3\)-equivariant projection
   \(\pi_{ij}:\ucx_{3}\to \ucx_{2}\)
   and we can
  define an associative convolution product between the categories:
  \[\star: \uMF_{\sigma,\tau}\ti \uMF_{\tau,\rho}\to \uMF_{\sigma,\rho},\]
  \begin{equation}\label{eq:big-con}
    \calF\star\calG:=\pi_{13*}(\CE_{\frn^{(2)}}(\pi^*_{12}(\calG)\oti\pi_{23}^*(\calG))^T).\end{equation}

  The smaller space \(\calX=\frg\ti G\ti \frn\ti G\ti \frn\) naturally embeds into the big space
  \[\underline{i}:\scX\to \ucx.\]

  The pull-back \(\underline{i}^*(W_{\sigma,\tau})\) is independent of \(\sigma,\tau\) and we denote this potential \(W\).
  The  corresponding matrix factorization category:
  \[\MF:=\MF_{G\ti B^2}(\scX,W)\]
  was studied in \cite{OblomkovRozansky16}. In particular, the pull-back morphism \(\ui^*:\uMF_{\bullet,\bullet}\to \MF\)
  intertwines the above defined convolution product with the product from \cite{OblomkovRozansky16}. In
  \cite{OblomkovRozansky16} we define a homomorphism from the braid group \(\Br_n\) to \((\MF,
  \star)\). It turns out that we can extend this homomorphism to the category \(\uMF_{\bullet,\bullet}\). It is easier to explain this
  homomorphism with the \ysh\ (that is, \Knr-reduced) category \(\uoMF_{\sigma,\tau}\) which we introduce in the next subsection.

\subsection{\ySh\ category}
\label{sec:short-category}

We define the short category as a category of \(B^2\)-equivariant matrix factorizations:
\[\uoMF_{\sigma,\tau}=\MF_{B^2}(\scXr,\Wr_{\sigma,\tau}), \quad \uscXr=\frb\ti G\ti \frb,\]
\begin{equation}
  \label{eq:pot2}
\Wr_{\sigma,\tau}(X,g,Y)=\Tr(X(\sigma\cdot Y-\Ad_g(\tau\cdot Y))).
\end{equation}

This category is related to the big category from the previous section by the Knorrer functor and the trivial  extension of \(G\)-action.
The  argument is parallel to the argument from \cite{OblomkovRozansky16}. First we introduce the intermediate  space \(\ucx^\circ\) which is
the \(G\)-quotient of the space \(\ucx\):
\[\scX^\circ\subset\frg\ti  G\ti \frb^2, \quad \scX^\circ=\{\dia(Y_1)=\dia(Y_2)\},\]
\[  W^\circ_{\sigma,\tau}(X,g,Y_1,Y_2)=\Tr(X(\sigma\cdot Y_1-\Ad_{g}(\tau\cdot Y_2))).\]

Since the \(G\)-action on \(\scX\) is free, the quotient map is the equivalence between the categories \(\uMF_{\sigma,\tau}\) and
\[\uMF_{\sigma,\tau}^\circ=\MF_{B^2}(\scX,W^\circ_{\sigma,\tau}).\]

Next we observe that we can write the potential as a sum of two terms with the second term being quadratic:
\[W^\circ_{\sigma,\tau}(X,g,Y_1,Y_2)=\Tr(X_{+}(\sigma\cdot Y_1-\Ad_g(\tau\cdot Y_2)))+\Tr(X_{--}((Y_1)_{++}+\Ad_g(Y_2)_{++})).\]
Since \(\dia(Y_1)=\dia(Y_2)\) the first term is equal to \(\Wr_{\sigma,\tau}\) and the second term is quadratic
with respect to the coordinates \(X_{--}\) and \((Y_1)_{++}+\Ad_g(Y_2)_{++}\). Thus we can apply the
Knorrer periodicity functor to establish the isomorphism of categories:
\[\MF(\underline{\scX}^\circ,W_{\sigma,\tau})\simeq \MF(\uscXr,\Wr_{\sigma,\tau}).\]
To upgrade this relation to the level of \(B^2\)-equivariant matrix factorizations we need to follow the method
of \cite{OblomkovRozansky16}.
In more details we need to define an auxiliary subspace \(\widetilde{\uscX}\) of \(\uscX\)
\[\widetilde{\uscX}=\frb\ti G\ti \frb\ti G\ti \frb,\quad j^x:\widetilde{\uscX}\to \uscX.\]
The auxiliary space \(\widetilde{\uscX}\) projects to the space \(\uscXr\):
\[\pi_y:\widetilde{\uscX}\to\uscXr,\quad \pi_y(X,g_1,Y_1,g_2,Y_2)=(X,g_1^{-1}g_2,Y_2).\]
There is a unique \(B^2\)-equivariant structure on \(\widetilde{\uscX}\) that makes map \(\pi_y\)-equivariant.
On the other hand the embedding \(j^x\) have enough \(B^2\)-equivariant properties to have well-defined
push-forward functor
\[j^x_*:\MF_{B^2}(\widetilde{\uscX},\pi^*_y(\Wr_{\sigma,\tau}))\to \MF_{G\ti B^2}(\uscX,W_{\sigma,\tau}).\]
Thus we can define the equivariant Knorrer functor as composition:
\[\Phi: \MF_{G\ti B^2}(\uscX,W_{\sigma,\tau})\to \MF_{B^2}(\uscXr,\Wr_{\sigma,\tau}),\quad \Phi=j_*^x\circ \pi_{y*}.\]

The  adjoint functor \(\Psi=\pi_{y*}\circ j^{x*}\) is the left inverse of \(\Phi\):
\begin{equation}\label{eq:psi-phi}
  \Psi\circ \Phi=1.\end{equation}

\subsection{Short convolutions}
\label{sec:short-convolutions}

The  short category has two convolution structures \(\stry,\strx\) which we will prove to be the same.
Both convolutions use the same  \(B^3\)-equivariant convolution space
\[\uscXr_3=\frb\ti G^3\ti \frb,\quad (b_1,b_2,b_3)\cdot (X,g_{12},g_{23},Y)=(\Ad_{b_1}(X),g_{12},g_{23},\Ad_{b_3}(Y)).\]

First we define the convolution that is transported from the big category by the Knorrer functor \(\Phi\). We
define the projection maps \(\bar{\pi}_{ij,y,\rho}:\uscXr_3:\to \uscXr\)
\[\bar{\pi}_{12,y,\rho}(X,g_{12},g_{23},Y)=(X,g_{12},\Ad_{g_{23}}(\rho\cdot Y)_{++}+\dia(Y)),\]
\[\bar{\pi}_{23,y,\rho}(X,g_{12},g_{23},Y)=(\Ad^{-1}_{g_{12}}(X)_+,g_{23},Y),\]
\[\bar{\pi}_{13,y,\rho}(X,g_{12},g_{23},Y)=(X,g_{12}g_{23},Y).\]

The construction of the convolution \(\stry\) is related to the convolution structure \(\bar{\star}\)
from \cite{OblomkovRozansky16} by the pull-back \(\underline{i}^*\). In particular, the main potential is the sum
\[\Wr_{\sigma,\tau}(X,g,Y)=\Wr+\delta\Wr_{\sigma,\tau},\]
where
\[
\Wr=-\Tr(X\Ad_g(Y_{++})),\qquad
 \delta\Wr_{\sigma,\tau}=\Tr(X\sigma\cdot\Delta(Y))-\Tr(X\Ad_g(\tau\cdot Y)).
\]

\begin{proposition}
  For any \(\sigma,\tau,\rho\) we have
  \[\bar{\pi}_{12,y,\rho}^*(\Wr_{\sigma,\tau})+\bar{\pi}_{23,y,\rho}^*(\Wr_{\tau,\rho})=\bar{\pi}_{13,y,\rho}^*(\Wr_{\sigma,\rho}).\]
\end{proposition}
\begin{proof}
  The pull-back along \(\underline{i}\)  of this equation yields a relation between the pull-backs of the potential
  \(\Wr\) and this relation is proven in proposition 5.3 of \cite{OblomkovRozansky16}. Thus we only need
  to show the relation between the pull-backs of \(\delta\Wr_{\bullet,\bullet}\), which is verified by a direct computation:
  \begin{multline*}
    \Tr(X\sigma \cdot \Delta(Y))-\Tr(X\Ad_{g_{12}}(\Ad_{g_{23}}(\rho\cdot \Delta(Y))_{++}+\tau\cdot \Delta(Y)))+
    \Tr(\Ad_{g_{12}}^{-1}(X)_+\tau\cdot\Delta(Y))-\\
    \Tr(\Ad_{g_{12}}^{-1}(X)_+\Ad_{g_{23}}(\rho\cdot\Delta(Y)))=
    \Tr(X\sigma \cdot \Delta(Y))-\Tr(\Ad^{-1}_{g_{12}}(X)(\Ad_{g_{23}}(\rho\cdot \Delta(Y))_{++}+\tau\cdot \Delta(Y)))+\\
    \Tr(\Ad_{g_{12}}^{-1}(X)\tau\cdot\Delta(Y))-\Tr(\Ad_{g_{12}}^{-1}(X)_+\Ad_{g_{23}}(\rho\cdot\Delta(Y)))=
    \Tr(X\sigma \cdot \Delta(Y))-\\\Tr(\Ad^{-1}_{g_{12}}(X)\Ad_{g_{23}}(\rho\cdot \Delta(Y))_{++})-\Tr(\Ad_{g_{12}}^{-1}(X)_+\Ad_{g_{23}}(\rho\cdot\Delta(Y)))= \Tr(X\sigma \cdot \Delta(Y))-\\\Tr(\Ad^{-1}_{g_{12}}(X)(\Ad_{g_{23}}(\rho\cdot \Delta(Y))-\Ad_{g_{23}}(\rho\cdot \Delta(Y))_-))-\Tr(\Ad_{g_{12}}^{-1}(X)\Ad_{g_{23}}(\rho\cdot\Delta(Y))_-)\\
    =\bar{\pi}_{13,y,\rho}^*(\delta \Wr_{\sigma,\rho}).
  \end{multline*}
\end{proof}

Thus for two matrix factorizations \(\calF\in\uoMF_{\sigma,\tau} \), \(\calG\in\uoMF_{\tau,\rho}\), the tensor product of
the pull-backs \(\bar{\pi}_{12,y,\rho}^*(\calF)\) and \(\bar{\pi}_{23,y,\rho}^*(\calG)\) is a
matrix factorization with the potential \(\bar{\pi}_{13,y,\rho}^*(\Wr_{\sigma,\rho})\). The  convolution
space \(\overline{\scX}_3=\frb\ti G^2\ti \frn\) from section 5 of \cite{OblomkovRozansky16} embeds
naturally inside the enlarged convolution space \(\underline{\scXr_3}\):
\[\underline{i}:\overline{\scX}_3\to\underline{\scXr}. \] The  morphisms \(\bar{\pi}_{ij}: \overline{\scX}_3\to\overline{\scX}_2\)
are intertwined by \(\underline{i}\):
\[\underline{i}\circ \bar{\pi}_{ij}=\bar{\pi}_{ij,y,\rho}\circ \underline{i}.\]

The  maps \(\bar{\pi}_{ij,y,\rho}\) are not \(B^3\)-equivariant but in section 5.4
of \cite{OblomkovRozansky16} we presented a construction for \(B^3\)-equivariant
enrichment of the tensor product of \(\bar{\pi}_{12}^*(\calF')\) and
\(\bar{\pi}_{23}^*(\calG')\):
\[\bar{\pi}_{12}^*(\calF')\otimes_B \bar{\pi}_{23}^*(\calG')\in \MF_{B^3}(\overline{\scX}_3,\bar{\pi}_{13}^*(\Wr)),\quad \calF',\calG'\in \MF.\]
This construction allows us to define a binary operation \cite{OblomkovRozansky16} on \(\overline{\MF}\):
\[\calF'\bar{\star}\calG'=\bar{\pi}_{13*}(\CE_{\frn^{(2)}}(\bar{\pi}_{12}^*(\calF')\otimes_B \bar{\pi}_{23}^*(\calG'))^{T^{(2)}}).\]

An almost verbatim repetition of the construction from the section 5.4 of
\cite{OblomkovRozansky16} yields  the desired extension of the construction
for the maps \(\pi_{ij,y,\rho}\):
\begin{proposition}
  There is a \(B^3\)-equivariant structure on the tensor product \(\bar{\pi}_{12,y,\rho}^*(\calF)\oti\bar{\pi}_{23,y,\rho}^*(\calG)\), \(\calF\in \uoMF_{\sigma,\tau}\), \(\calG\in \uoMF_{\tau,\rho}\), which we denote
  \[\bar{\pi}_{12,y,\rho}^*(\calF)\oti_B\bar{\pi}_{23,y,\rho}^*(\calG)\in \MF_{B^3}(\underline{\scXr_3},\bar{\pi}_{13,y,\rho}^*(\Wr_{\sigma,\rho})),\]
  such that
  \begin{equation}\label{eq:I-conv}
  \und{i}^*(\bar{\pi}_{12,y,\rho}^*(\calF)\oti_B\bar{\pi}_{23,y,\rho}^*(\calG))=
  \bar{\pi}_{12}^*(\und{i}(\calF))\oti_B\bar{\pi}_{23}^*(\und{i}(\calG)),
\end{equation}
\begin{equation}
  \label{eq:Kn-conv}
  \Phi\circ\bar{\pi}_{13,y,\rho,*}(\CE_{\frn^{(2)}}(\bar{\pi}_{12,y,\rho}^*(\calF)
  \oti_B\bar{\pi}_{23,y,\rho}^*(\calG))^{T^{(2)}})=
  \pi_{13,*}(\CE_{\frn^{(2)}}(\pi_{12}^*(\calF')
  \oti\pi_{23}^*(\calG'))^{T^{(2)}}),
\end{equation}
where \(\calF'=\Phi(\calF)\), \(\calG'=\Phi(\calG)\).
\end{proposition}

Using the previous proposition we define the binary operation:
\[\stry: \uoMF_{\sigma,\tau}\ti\uoMF_{\tau,\rho}\to \uoMF_{\sigma,\rho},\]
 \begin{equation}\label{eq:shrt-con}\calF\stry\calG=\bar{\pi}_{13,y,\rho,*}(\CE_{\frn^{(2)}}(\bar{\pi}_{12,y,\rho}^*(\calF)
  \oti_B\bar{\pi}_{23,y,\rho}^*(\calG))^{T^{(2)}}).\end{equation}

The  equation \eqref{eq:Kn-conv} together with \eqref{eq:psi-phi} implies that the operation is associative.
On the other hand, equations \eqref{eq:I-conv} and \eqref{eq:Kn-conv} imply that \(\und{i}^*\) and \(\Phi\) are
homomorphisms of the convolution algebras.

\subsection{Induction functors}
\label{sec:induction-functors}

For a parabolic subgroup \(P\subset G_n\)  define auxilary spaces:
\[\und{\scX_2}^\circ(P)=\frp\ti P\ti \frb^2,\quad \und{\scXr_2}(P)=\frb\ti P\ti \frb,\]
where \(\frp=\mathrm{Lie}(P)\). We denote \(G_n=\GL(n)\), and \(P_k\subset G_n\)
is the parabolic subgroup with Lie algebra generated by \(\frb\) and \(E_{k,1}\).

Respectively, we denote by \(p_k\) the natural projection from \(P_k\) to \(G_k\ti G_{n-k}\)
and we use the same notation the corresponding projection of the Lie algebras. We also use
notation \(i_k\) for the embedding \(P_k\to G_n\) and its Lie algebra cousins.
These projections and inclusions
induces the \(B^2_n\)-equivariant morphisms:
\[p_k:\und{\scX_2}^\circ(P_k)\to \und{\scX_2}^\circ(G_k)\ti \und{\scX_2}^\circ(G_{n-k}),\quad
  \bar{p}_k:\und{\scXr_2}(P_k)\to \und{\scXr_2}(G_k)\ti \und{\scXr_2}(G_{n-k})\]
\[i_k:\und{\scX_2}^\circ(P_k)\to  \und{\scX_2}^\circ(G_{n}),\quad
  \bar{i}_k:\und{\scXr_2}(P_k)\to \und{\scXr_2}(G_n).\]

In the setting of symmetric group we have analogous notions of parabolic subgroups and inclusion
homomorphism  \(i_k: S_k\ti S_{n-k}\to S_n\). In close analogy to constructions from section 6 of \cite{OblomkovRozansky16}
we introduce the induction functors:
\[\ind_k: \MF_{B_k^2}(\und{\scX_2}^\circ(G_k),W_{\sigma_1,\tau_1}^\circ)\ti
  \MF_{B_{n-k}^2}(\und{\scX_2}^\circ(G_{n-k}),W_{\sigma_2,\tau_2}^\circ)\to \MF_{B_n^2}(\und{\scX_2}^\circ(G_n),W_{\sigma,\tau}^\circ) \]
\[\overline{\ind_k}: \MF_{B_k^2}(\und{\scXr_2}(G_k),\Wr_{\sigma_1,\tau_1})\ti
  \MF_{B_{n-k}^2}(\und{\scXr_2}(G_{n-k}),\Wr_{\sigma_2,\tau_2})
  \to \MF_{B_n^2}(\und{\scXr_2}(G_n),\Wr_{\sigma,\tau}) \]
where \(\sigma=i_k(\sigma_1,\sigma_2),\tau=i_k(\tau_1,\tau_2)\).

The  arguments of section 6 of \cite{OblomkovRozansky16} provide a proof of the key properties
of the induction functors.
\begin{proposition}
  We have
  \[\ind_k\circ (\Phi_k\ti \Phi_{n-k})=\Phi_n\circ \overline{\ind_k}\]
\end{proposition}

\begin{proposition}
  For any
  \[\calF_i,\calG_i\in \MF_{B_k^2}(\und{\scX_2}^\circ(G_k),W_{\sigma_{1i},\tau_{1i}}^\circ),\quad
    \calF_2,\calG_2\in \MF_{B_{n-k}^2}(\und{\scX_2}^\circ(G_{n-k}),W_{\sigma_{2i},\tau_{2i}}^\circ),\]
  we have
  \[\ind_k(\calF_1,\calF_2)\star \ind_k(\calG_1,\calG_2)=\ind_k(\calF_1\star\calG_1,\calG_1\star\calG_2)\]
  where \(\tau_{1i}=\sigma_{2i}\).
\end{proposition}

The  combination of these two propositions also implies that the induction functor \(\overline{\ind_k}\)
is a homomorphism of  \(\stry\)-convolution algebras.
We leave the detailed discussion of the induction functors for the forth-coming
publication \cite{OblomkovRozansky18c} where  we explain the relation between the induction functors and the
elliptic Hall algebra.

\subsection{Units of convolution algebras}
\label{sec:units-conv-algebr}

The potentials \(W_{\tau,\tau}\) and \(\Wr_{\tau,\tau}\) vanish on subvarieties
\(\und{\scX_2}(B)\) and \(\und{\scXr_2}(B)\) thus we can define Koszul matrix factorizations:
\[\und{\calc}_{\parallel}=i_{B*}(\CC[\und{\scX_2}(B)]),\quad \und{\calcr}_{\parallel}=\bar{i}_{B*}(\CC[\und{\scXr_2}])
\]
where \(i_B\) and \(\bar{i}_B\) are the corresponding embeddings.

The argument of proposition 7.1 and corollary 7.2 from \cite{OblomkovRozansky16} implies that these matrix factorizations
are units in the convolution algebras.

\begin{proposition}
  The matrix factorizations \(\und{\calc}_\parallel\in \uMF_{\tau,\tau}\)  and \(\und{\calcr}_\parallel\uoMF_{\tau,\tau}\) are the units in the convolution algebras.
  That is for any \(\calF\in \uMF_{\tau,\sigma},\calG\in\uMF_{\sigma,\tau}\), \(\bar{\calF}\in \uoMF_{\tau,\sigma},
  \bar{\calG}\in \uoMF_{\sigma,\tau}\) we have:
  \[\und{\calc}_\parallel\star\calF=\calF,\quad \calG\star\und{\calc}_\parallel=\calG,\]
  \[\und{\calcr}_\parallel\stry\bar{\calF}=\bar{\calF},\quad \bar{\calG}\stry\und{\calcr}_\parallel=\bar{\calG}.\]
\end{proposition}

Choose a basis in the group of characters \(B^\vee\) of \(B\subset \GL_n\): \(\chi_k(b)=b_{kk}\), \(k=1,\dots,n\). RFr a matrix \(B^2\) factorization \(\calF\) we denote by \(\calF\langle \chi',\chi''\rangle\) the matrix factorization with
the twisted \(B^2\)-action. In particular, we can consider the twisted version of the units. The  same argument as before
implies that these elements form a large commutative subalgebra related to the algebra of Jucys-Murphy
elements \cite{OblomkovRozansky17}.

\begin{proposition}
  The elements \(\und{\calc}_\parallel\langle \chi',\chi''\rangle\), \(\und{\calcr}_\parallel\langle \chi',\chi''\rangle\),
  \(\chi',\chi''\in B^\vee\) mutually commute. The elements \(\und{\calc}_\parallel\langle 1,0\rangle\), \(\und{\calcr}_\parallel\langle 1,0\rangle\) are the central elements of the convolution algebras.
\end{proposition}

Proposition 9.1 of \cite{OblomkovRozansky16} implies that a twist of the unit by a pair of the characters \(\chi',\chi''\) depends only on their sum:
\[\und{\calc}_\parallel\langle\chi',\chi''\rangle\sim\und{\calc}_{\parallel}\langle \chi'+\beta,\chi''-\beta\rangle,\quad \und{\calcr}_\parallel\langle\chi',\chi''\rangle\sim\und{\calcr}_{\parallel}\langle \chi'+\beta,\chi''-\beta\rangle.\]
\section{Braid group action}
\label{sec:braid-group-action}

In this section we adjust the results of \cite{OblomkovRozansky16} in order to obtain the
homomorphisms from the braid group to the convolution algebras:
\[\und{\Phi^{br}}: \Br_n\to \MF_{\bullet,\bullet},\quad \und{\overline{\Phi^{br}}}: \Br_n\to
  \uoMF_{\bullet,\bullet}.\]

First, we construct the matrix factorizations for elementary two-strand braids, and then extend the homomorphism to arbitrary braids through the induction functor, verifying the braid relations on three-strand braids.

\subsection{Braid groups}
\label{sec:braid-groups}

The braid group \(\Br_n\) is generated by elements \(\sigma(i)\), \(i=1,\dots,n-1\) with relations:
\[\sigma(i)\sigma(i+1)\sigma(i)=\sigma(i+1)\sigma(i)\sigma(i+1),\quad i=1,\dots,n-2.\]
There is a natural surjective homomorphism:
\[\Sigma: \Br_n\to S_n\]
with kernel consisting of elements of the group of pure braids \(\mathfrak{PBr}_n\). Let us abbreviate
the projection by
\(\Sigma(\tau)=\dot{\tau}\).

It is convenient for us to work with labeled braids. These braids form a groupoid \(\LBr_n\) which could be described
as a subset of \(S_n\ti \Br_n \ti S_n\):
\[ (s,\sigma,t)\in \LBr_n\mbox{ iff } \dot{\sigma}t=s.\]
We use notation \({}_s\sigma_t\) for \((s,\sigma,t)\in \LBr_n\). The elements \({}_s\sigma_t\) and \({}_r\tau_p\)
are composable if \(t=r\) and \({}_s\sigma_t\cdot {}_t\tau_p={}_s\sigma\tau_p\).

\subsection{Two strand case}
\label{sec:two-strand-case}
Let \(\tau\) be a generator of \(S_2\) and \(\sigma_1\) is the positive generator of \(\Br_2\).
Let us also fix the coordintes on the space \(\und{\scXr_2}\):
\[X=
  \begin{bmatrix}
    x_{11}& x_{12}\\ 0&x_{22}
  \end{bmatrix},\quad
  Y=
  \begin{bmatrix}
    y_{11}&y_{12}\\ 0&y_{22}
  \end{bmatrix},\quad
  g=
  \begin{bmatrix}
    a_{11}&a_{12}\\
    a_{21}&a_{22}
  \end{bmatrix}
 \]
 Since the potential \(\Wr_{1,\tau}=\Wr_{\tau,1}\) factors
 \[\Wr_{1,\tau}(X,g,Y)=\tilde{x}_0\tilde{y}_0,\quad \tilde{x}_0=(x_{11}-x_{22})a_{11}+x_{12}a_{21},\quad
 \tilde{y}_0=((y_{11}-y_{22})a_{22}-y_{12}a_{21})/\det(g).\]

Thus we can define the Koszul matrix factorizations with the potentials \(\Wr_{\tau,1}\), \(W_{\tau,1}\):
\[\und{\calcr}_+=\mathrm{K}^{\Wr_{\tau,1}}(\tilde{x}), \quad \und{\calc}_+=\Phi(\und{\calcr}_+).\]

We use same notation for the corresponding matrix factorizations from \(\uMF_{1,\tau}\) and \(\uoMF_{1,\tau}\).

Following blue-prints of \cite{OblomkovRozansky16} we define the inverses of the above matrix factorizations:
\[\und{\calcr}_-=\und{\calc}_+\langle-\chi_1,\chi_2\rangle,\quad \und{\calcr}_-=\und{\calcr}_+\langle-\chi_1,\chi_2\rangle. \]

The  agruments of section 3.3 \cite{OblomkovRozansky17} imply that the  we can switch \(\chi_1\) and \(\chi_2\) in the last
formula:
\[\und{\calc}_-=\und{\calc}_+\langle\chi_2,-\chi_1\rangle,\quad \und{\calcr}_-=\und{\calcr}_+\langle\chi_2,-\chi_1\rangle. \]

These matrix factorizations correspond to the positive and negative crossings on two strands:
\begin{proposition}
  We have
  \[\und{\calc}_+\star\und{\calc}_-\sim\und{\calc}_-\star\und{\calc}_+\sim \und{\calc}_\parallel,\quad
  \und{\calcr}_+\star\und{\calcr}_-\sim\und{\calcr}_-\star\und{\calcr}_+\sim \und{\calcr}_\parallel.\]
\end{proposition}

\subsection{Braid group generators}
\label{sec:braid-group-gener}

Now we can describe the generators of the braid group and state our main result of the section. Indeed, let us define variations of
the induction functors from the previous section:
\[\ind_{k,k+r}:\uMF(\und{\scX_2}(G_{r+1}),W_{\sigma,\tau}\to\uMF(\und{\scXr_2}(G_n),W_{\sigma,\tau}) ),\] \[\overline{\ind}_{k,k+r}:\uMF(\und{\scXr_2}(G_{r+1}),\Wr_{\sigma,\tau})\to \uMF(\und{\scXr_2}(G_n),\Wr_{\sigma,\tau}),\]
as  compositions:
\[\ind_{k,k+r}(\calF)=\ind_k(\und{\calc}_\parallel,\ind_r(\calF,\und{\calc}_\parallel)),\quad
  \overline{\ind}_{k,k+r}(\bar{\calF})=\overline{\ind}_k(\und{\calcr}_\parallel,\overline{\ind}_r(\bar{\calF},\und{\calcr}_\parallel))\]
With these functors we define the generators:
\[\und{\calc}_\pm^{(i)}=\ind_{i,i+1}(\und{\calc}_\pm),\quad \und{\calcr}_\pm^{(i)}=\ind_{i,i+1}(\und{\calcr}_\pm).\]

The  elements \({}_{\tau_i}\sigma(i)_1\) and \({}_1\sigma(i)_{\tau_i}\) generate the groupoid \(\LBr_n\), hence its representation is determined by their images:
\begin{theorem}\label{thm:conv}
  The assignments:
  \[{}_{\tau_i}\sigma^{\pm 1}(i)_1\mapsto \und{\calc}_\pm^{(i)}, \quad {}_1\sigma^{\pm 1}(i)_{\tau_i}\mapsto\und{\calc}_\pm^{(i)}, \]
  \[{}_{\tau_i}\sigma^{\pm 1}(i)_1\mapsto \und{\calcr}_\pm^{(i)},\quad  {}_1\sigma^{\pm 1}(i)_{\tau_i}\mapsto\und{\calcr}_\pm^{(i)}, \]
  extend to groupoid homomorphisms:
  \[\Phi^{\lbr}:\LBr_n\to \uMF_{\bullet,\bullet},\quad\bar{\Phi}^{\lbr}:\LBr_n\to \uoMF_{\bullet,\bullet}.\]
\end{theorem}

Recall a technical tool for computations with matrix factorizations from \cite{OblomkovRozansky16}.
Denote by \(G_{S}\subset G_n\), \(S\subset \{1,\dots,n\}\)
the subgroup whose Lie algebra is generated by \(\frb\) and \(E_{ij}\), \(i>j,i,j\in S\). Next, define
smaller convolution spaces \(\und{\scX_3}(G_S,G_T)=\frg\ti G_S\ti G_T\ti \frb\), \(\und{\scX_2}(G_S)=\frg\ti G_S\ti G \ti \frb\) .
The  transitivity of the induction functors implies that
\begin{equation}\label{eq:ind-conv}
   \calG \stry\ind_{k,k+r}(\calF)=\bar{\pi}_{13,y,\rho*}(\CE_{\frn^{(2)}_{r+1}}(\bar{\pi}_{12,y,\rho}^*(\calG)\oti_{B_{r+1}}\bar{\pi'}_{23,y,\rho}^*(\calF))).
\end{equation}
where the maps \(\bar{\pi}_{13,y,\rho}\) and \(\bar{\pi}_{12,y,\rho}\) are the restrictions of the
corresponding maps to the space \(\und{\scX_3}(G_{\{k,\dots,k+r\}},G)\) and \(\bar{\pi'}_{23,y,\rho}\) is the composition
of \(\bar{\pi}_{23,y,\rho}\) with a natural projection \(\und{\scX_2}(G_{\{k,\dots,k+r\}})\to \und{\scX_2}(G_r)\).

\begin{proof}[Proof of theorem~\ref{thm:conv}]
We have  to show that the matrix factorization \(\und{\calc}_\pm^{(i)}\) satisfy the braid relations for
  the three-stranded braids:
  \begin{equation}
    \label{eq:br-rel}
    \und{\calc}_+^{(1)}\stry \und{\calc}_+^{(2)}\stry\und{\calc}_+^{(1)}\sim \und{\calc}_+^{(2)}\stry \und{\calc}_+^{(1)}\stry\und{\calc}_+^{(2)}
  \end{equation}
together with the relation
  \begin{equation}
    \label{eq:br-inv}
    \und{\calc}_+\stry \und{\calc}_-\sim \und{\calc}_\parallel.
  \end{equation}
  which says that $\und{\calc}_-$ is the inverse of $\und{\calc}_+$ with respect to the convolution $\stry$.
  The relations in the multi-strand case follow from the properties of the induction functors.

First, we prove \eqref{eq:br-inv}. Observe that both \(\und{\calc}_+\)  and \(\und{\calc}_-\) are
  matrix factorizations that extend the Koszul complexes:
  \[\und{\calcr}_+^+=\mathrm{K}[\tilde{x}],\quad \und{\calcr}_-^+=\mathrm{K}[\tilde{x}]\langle \chi_2,-\chi_1\rangle. \]
  In particular, the differentials in these complexes do not depend on the entries of \(Y\)-matrix. Hence Theorem 9.6
  from \cite{OblomkovRozansky16} implies that
\[\bar{\pi}_{13,y,\rho,*}(\CE_{\frn^{(2)}}(\bar{\pi}_{12,y,\rho}^*(\und{\calcr}_+^+)
  \oti_B\bar{\pi}_{23,y,\rho}^*(\und{\calc}_-^+))^{T^{(2)}})\sim \mathrm{K}[g_{21}].\]

Finally, observe that since \(g_{21}\) is an irreducible polynomial inside the structure ring, the matrix
factorization \(\und{\calcr}_\parallel=\mathrm{K}^{\Wr_{x,x}}[g_{21}]\) is a unique extension of \(\mathrm{K}[g_{21}]\) by Lemmas 3.1, 3.2
of \cite{OblomkovRozansky16}. Thus the equation \eqref{eq:br-inv} follows.

In our discussion of braid relations we assume that both sides of \eqref{eq:br-rel} are objects of the category
\(\uoMF_{1,c}\), \(c=t_{12}t_{23}t_{12}=t_{23}t_{12}t_{23}\).
To prove the braid relation we need to combine the previous construction with the formula \eqref{eq:ind-conv}. Indeed, first
we argue that
\begin{equation}\label{eq:rel12}
  \und{\calcr}_+^{(1)}\stry\und{\calcr}_+^{(2)}\sim \und{\calcr}_{12},\quad \und{\calcr}_+^{(2)}\stry\und{\calcr}_+^{(1)}\sim \und{\calcr}_{21}.\end{equation}
\begin{equation}\label{eq:rel21}\und{\calcr}_{12}=\mathrm{K}^{\Wr_{\bullet,\bullet}}[f,g,a_{31}], \quad \und{\calcr}_{21}=\mathrm{K}^{\Wr_{\bullet,\bullet}}[h,k,(a^{-1})_{31}],\end{equation}
\[f=(x_{11}-x_{22})a_{11}+x_{12}a_{21},\quad g=(a^{-1})_{33}(x_{11}-x_{33})-(a^{-1})_{32}x_{32}-(a^{-1})_{31}x_{13}\]
\[ h=(a^{-1})_{32} x_{23}+(a^{-1})_{33}(x_{33}-x_{22}),\quad k=a_{11}(x_{11}-x_{33})+a_{21}x_{12}+a_{31}x_{13}.\]

We prove only the first formula of \eqref{eq:rel12}, the proof of the second one is similar.
Consider the formula \eqref{eq:ind-conv} for \(k=2,r=1\), \(\calF=\und{\calcr}_+\), \(\calG=\und{\calcr}_+^{(1)}\).
We have \(\und{\calcr}_+^{(1)}=\mathrm{K}^{\Wr_{\bullet,\bullet}}[\tilde{x},a_{31},a_{32}]\) where \(a_{ij}\) are coordinates on \(G\)
in the product \(\und{\scX_2}=\frb\ti G\times\frb\). The \(\und{\calcr}_+^{(1)+}=\mathrm{K}[\tilde{x},a_{31},a_{32}]\) is regular hence
\(\und{\calcr}_+^{(1)+}\) is a unique lift of the complex to the matrix factorization.

Next we observe that Lemma 10.4 of \cite{OblomkovRozansky16} implies that
\[\und{\calcr}^{(1)+}_+\stry \ind_{2,3}(\und{\calcr}_+^+)\sim \und{\calcr}_{12}^+,\]
where \(\und{\calcr}_{12}^+=\mathrm{K}[f,g,a_{31}]\). By proposition 10.2 the sequence \(f,g,a_{31}\) is regular in the
structure ring. Since \(\Wr_{1,t_{12}t_{23}}\in (f,g,a_{31})\) by Lemmas 3.1, 3.2 there is a unique extension of the complex
\(\und{\calcr}_{12}^+\) to a matrix factorization with the corresponding potential and the equation \eqref{eq:rel12} follows.

To complete our proof we need to use \eqref{eq:ind-conv} again for \(\calG=\und{\calcr}_{12}\) and \(\calF=\und{\calcr}_+\) and \(k=1,r=1\).
It is shown in Lemma 11.5 of \cite{OblomkovRozansky16} that formula \eqref{eq:ind-conv} can be used to show
\[\und{\calcr}_{12}^+\stry\und{\calcr}_+^{(1)+}\sim \und{\calcr}_{121}^+, \]
where \(\und{\calcr}_{121}^+\) is a complex that is quasi-isomorphic to a module in zeroth homological degree which is
annihilated by the potential \(\Wr_{1,c}\).

It is also shown in corollary 11.6 of \cite{OblomkovRozansky16} that
\[\und{\calcr}_{21}^+\stry\und{\calcr}_+^{(2)+}\sim \und{\calcr}_{212}^+, \]
and \(\und{\calcr}_{212}^+\) is isomorphic to \(\und{\calcr}_{121}^+\) hence by the uniqueness of the extension lemma the braid relations
follow.
\end{proof}

In the rest of the paper we use notations \(\und{\calc}_\alpha\), \(\und{\calcr}_\alpha\), \(\alpha\in \LBr_n\) for the corresponding
matrix factorization.

\section{Fourier transform}
\label{sec:involution}

\subsection{Specialization to unlabeled braids.}
\label{sec:spec-unlab-braids}

The  action of the group \(S_n\) on the last factor of \(\und{\scXr_2}\) induces the isomorphism of the categories:
\[\tau\cdot: \uoMF_{\sigma,\rho}\to \uoMF_{\sigma\cdot\tau,\rho\cdot\tau}.\]

Let us choose a special element in the orbit of the action:
\[\uoMF_\sigma:=\uoMF_{\sigma,1}.\]

We denote by \(\uoMF_\cdot\) the union of such categories. This union has a natural monodial structure:
\[\calF\stty\calG:=\calF\stry (\tau^{-1}\cdot \calG),\quad \calG\in \uoMF_\tau.\]

Theorem~\ref{thm:conv} implies that there is a homomorphism from the braid group to the
last monoidal category:

\begin{corollary}
The assignments:
  \[\sigma^{\pm 1}(i)\mapsto \und{\calcr}_\pm^{(i)}, \]
  extend to the group homomorphism:
  \[\Phi^{\br}:\Br_n\to (\uoMF_{\bullet,\bullet}, \stty).\]
\end{corollary}

\subsection{Fourier morphisms}
\label{sec:fourier-morphisms}

There is a natural automorphism on the small space \(\und{\scXr}_2\)
defined by:
\[\FT(X,g,Y)\to (Y,g^{-1}, X).\]

This involution extends to the functor of the triangulated categories:
\[\FT:\uoMF_{\sigma}\to \uoMF_{\tau^{-1}}.\]

The key observation of this paper is that the automorphism, which we call Fourier transform, respects the monoidal structure
of the category \(\uoMF_{\bullet}\):
\begin{theorem}\label{thm:lin}
  For any composable \(\alpha,\beta\in\Br_n\) we have:
  \[
\FT\left(    \FT(\und{\calcr}_\alpha)\stty\FT(\und{\calcr}_\beta)\right)\sim\und{\calcr}_{\beta}\stty \und{\calcr}_{\alpha}.\]
\end{theorem}

  To start our argument we need to map out the morphisms that are used in both sides of the equation.
  In particular, the maps for the LHS fit into the diagram:
  \begin{equation}\label{eq:diagF}
    \begin{tikzcd}
    (X,g_{12},g_{23},Y)\arrow[r,dashed,"\pi_{13,y,\tau}"]\arrow[d,"\pi_{12,y,\tau}"]\arrow[rd,"\pi_{23,y,\tau}"]
    &(X,g_{13},Y)\arrow[r,dashed,"\tau^{-1}\cdot"]&(X,g_{13},\tau\cdot Y)
    \arrow[d,dashed,"\mathfrak{F}"]\\
    (X,g_{12},\Ad_{g_{23}}(\tau\cdot Y)_{++}+\Delta(Y))\arrow[d,"\mathfrak{F}"]&(\Ad_{g_{12}}^{-1}(X)_+,g_{23},Y)\arrow[d,"\tau^{-1}"]& (\tau\cdot Y,g_{13}^{-1},X) \\
    (\Ad_{g_{23}}(\tau\cdot Y)_{++}+\Delta(Y),g_{12}^{-1},X)\arrow[d]    &(\Ad_{g_{12}}^{-1}(X)_+,g_{23},\tau\cdot Y)\arrow[d,"\mathfrak{F}"]&\\
    (\Ad_{g_{23}}(\tau\cdot Y)_{++}+\Delta(Y),g_{12}^{-1},X)    &(\tau\cdot Y,g_{23}^{-1},\Ad_{g_{12}}^{-1}(X)_+)&
  \end{tikzcd},
\end{equation}
where \(g_{13}=g_{12}g_{23}\).
To perform the convolution for LHS of the equation we need to place \(\und{\calcr}_\alpha\in \MF_{\sigma}\) and \(\und{\calcr}_\beta\in\MF_{\tau^{-1}}\)
at the bottom of the diagram then pull-back along the solid arrows going up and push-forward along the dashed arrows. The  same algorithm applies for the RHS of the equation but with the diagram:
\begin{equation}\label{eq:diagNoF}
  \begin{tikzcd}
    (X,g_{12},g_{23},Y)\arrow[r,dashed,"\bar{\pi}_{13,y,\sigma^{-1}}"]\arrow[d,"\bar{\pi}_{23,y,\sigma^{-1}}"]\arrow[dr,"\bar{\pi}_{12,y,\sigma^{-1}}"]
    &(X,g_{13},Y)\arrow[r,dashed,"\sigma\cdot"]& (X,g_{13},\sigma^{-1}\cdot Y)\\
    (\Ad^{-1}_{g_{12}}(X)_+,g_{23},Y )\arrow[d,"\sigma\cdot"]&(X,g_{12},\Ad_{23}(\sigma^{-1}\cdot Y)_{++}+\Delta(Y))\arrow[d] &\\
    (\Ad^{-1}_{g_{12}}(X)_+,g_{23},\sigma^{-1}\cdot Y )&  (X,g_{12},\Ad_{23}(\sigma^{-1}\cdot Y)_{++}+\Delta(Y))&
  \end{tikzcd}
\end{equation}

If we apply to the diagram \eqref{eq:diagF} the change of variable below we will get a diagram that is close to \eqref{eq:diagNoF}:
\begin{equation}
  \label{eq:change-var}
  X\to \sigma^{-1}\cdot Y,\quad \tau\cdot Y\to X,\quad g_{12}\to g_{23}^{-1}, \quad g_{23}\to g_{12}^{-1}.
\end{equation}
The difference between the transformed diagram \eqref{eq:diagF} and \eqref{eq:diagNoF} is the compositions of the solid
arrows. Let us fix notations for the composition of the solid arrows in diagram \eqref{eq:diagNoF}:
\[\tilde{\pi}_{23,y,\sigma^{-1}}(X,g_{12},g_{23},Y)=  (\Ad^{-1}_{g_{12}}(X)_+,g_{23},\sigma^{-1}\cdot Y ),\]
\[\tilde{\pi}_{12,y,\sigma^{-1}}(X,g_{12},g_{23},Y)=(X,g_{12},\Ad_{23}(\sigma^{-1}\cdot Y)_{++}+\Delta(Y)).\]
Let us also fix notation \(\tilde{\pi}_{13,y,\sigma^{-1}}\) for composition of the dashed arrows.
Respectively, the composition of the solid arrows in \eqref{eq:diagF} after transformation
\eqref{eq:change-var} are:
\[\tilde{\pi}_{23,x,\sigma^{-1}}(X,g_{12},g_{23},Y)=(\Ad_{g_{12}}^{-1}(X)_{++}+\Delta(\tau^{-1}\cdot X),g_{23},\sigma^{-1}\cdot Y)\]
\[\tilde{\pi}_{12,x,\sigma^{-1}}(X,g_{12},g_{23},Y)=(X,g_{12},\Ad_{g_{23}}(\sigma^{-1}\cdot Y)_+).\]
Similarly, we denote by \(\tilde{\pi}_{13,x,\sigma^{-1}}\) the composition of the dashed arrows in the diagram  \eqref{eq:diagF}.


Thus our theorem would follow from the equation
\begin{equation}\label{eq:Four}
    \tilde{\pi}_{12,x,\sigma^{-1}}^*(\calF)
    \oti_B\tilde{\pi}_{23,x,\sigma^{-1}}^*(\calG)\sim
    \tilde{\pi}_{12,y,\sigma^{-1}}^*(\calF)
    \oti_B\tilde{\pi}_{23,y,\sigma^{-1}}^*(\calG)
  \end{equation}
  for \(\calF\in\uoMF_\tau\), \(\calG\in\uoMF_\sigma\).
  The  homotopy implies the relation between the two products:
  \[\calF\stty\calG\sim \calF\sttx\calG,\]
  where the product \(\sttx\) is defined with the use of maps \(\tilde{\pi}_{ij,x,\bullet}\).

  It is possible to prove the  formula~\eqref{eq:Four} in the full generality but the formula
  for the homotopy is rather complicated. The  details will appear in the future
  publications and in this note we provide explanations for the simplest possible
  case of the formula, and explain why it is enough for our proof.

  \begin{lemma}\label{lem:lin}
    Let us assume that \(\calF=(M',D',\partial')\in \uoMF_{\tau}\),
    \(\calG=(M'',D'',\partial'')\in \uoMF_{\tau}\) and that the differentials
    \(D',\partial',D'',\partial''\) are linear along the factors \(\frb\) in
    \(\und{\scXr_2}\). Then there is a  homotopy as in \eqref{eq:Four}.
\end{lemma}

\begin{proof}
  We can write an explicit formula for the homotopy. To simplify notations, we
  assume that all equivariant correction differentials vanish.
  We also identify \(\calF\), \(\calG\) with the pull-backs
  \(  \tilde{\pi}_{12,x,\sigma^{-1}}^*(\calF)\),\(\tilde{\pi}_{23,x,\sigma^{-1}}^*(\calG)\).

  The convolution space \(\und{\scX_3}=\frb\ti G\ti G\ti\frb\) has a standard coordinate
  system \(X,g_{12},g_{23},Y\) and two other coordinate systems pulled back along the
  maps \(\tilde{\pi}_{12,x,\sigma^{-1}}\) and \(\tilde{\pi}_{23,x,\sigma^{-1}}\). We denote the last two coordinate system by \((X',g_{12},g_{23},Y')\) and \(X'',g_{12},g_{23},Y''\).
  For  example \(Y'=\Ad_{g_{12}}(\sigma^{-1}\cdot Y)_+\).

  By taking the derivatives
  of the potentials \(\Wr_{\sigma,1}\),
   \(\Wr_{\tau,1}\) we get:
   \begin{equation}\label{eq:derW}
   [\partial_{Y'_{ii}}D',D']=\tau^{-1}\cdot X'_{ii}-(\Ad_{g_{12}}^{-1}X')_{ii},\quad
     [\partial_{X''_{ii}}D'',D'']=\sigma\cdot Y''_{ii}-(\Ad_{g_{23}}Y'')_{ii}.\end{equation}
   Since all differentials and the potentials are \(\frb\)-linear by taking the second
   derivatives of the potentials along \(\frb\) we prove that operators
   \begin{equation}\label{eq:def-u}
     u_i=\partial_{Y'_{ii}}D'\oti\partial_{X''_{ii}}D'',\end{equation}
   satisfy relations
   \begin{equation}\label{eq:rel-u}
     u_i^2=0,\quad u_iu_j+u_ju_i=0.
   \end{equation}
   Hence we can define the homotopy map by
   \[U=\exp(-\sum_{i=1}^n u_i).\]
   Combining all previous observations and taking into account
   \(X'=X\), \(Y''=\sigma^{-1} \cdot Y\) for \eqref{eq:derW} we arrive to
   \[U(D'\otimes I+I\oti D'') U^{-1}=\tilde{D'}\oti I+I\oti\tilde{D''},\]
   where \(\tilde{D'}\) and \(\tilde{D''}\) are the differentials of the pull-backs
   \(  \tilde{\pi}_{12,y,\rho}^*(\calF)\),\(\tilde{\pi}_{23,,\rho}^*(\calG)\).
\end{proof}

\begin{proof}[Proof of theorem~\ref{thm:lin}] In the proof we omit the subindex from \(S_n\) from the
  notations of the maps to simplify notations.
  Let \(\omega=\sigma^{\eps_1}(i_1)\dots \sigma^{\eps_l}(i_l)\) is a presentation of a braid
  as product of the elementary braids. We claim that there is a homotopy which implies the
   statement of theorem:
  \[\und{\calcr}_{\eps_1}^{(i_1)}\sttx\dots\sttx \und{\calcr}_{\eps_l}^{(i_l)}\sim
    \und{\calcr}_{\eps_1}^{(i_1)}\stty\dots\stty \und{\calcr}_{\eps_l}^{(i_l)}.\]

Indeed, we can iterate formula \eqref{eq:ind-conv} to obtain a formula for the
long product:
\begin{multline}\label{eq:long-prod}   \calG \stty\ind_{i_1,i_1+1}(\calF_1)\stty\dots \stty\ind_{i_l,i_l+1}(\calF_l)
  =\\
  \tilde{\pi}_{1l+2,y,\bullet*}(\CE_{\frn^l_{2}}(\tilde{\pi}_{12,y,\bullet}^*(\calG)\oti_{B_{2}}\tilde{\pi'}_{23,y,\bullet}^*(\calF_1)\oti_{B_{2}}\dots\oti_{B_2}\tilde{\pi'}^*_{l+1,l+2,y,\bullet}(\calF_l))),
\end{multline}
where \(\tilde{\pi}_{1,l+2,y,\bullet}\) is the map from \[\und{\scXr_{l+2}}(G, G_{i_1,i_1+1},\dots,G_{i_l,i_l+1})=
  \frb\ti G_{i_1,i_1+1}\ti\dots\ti G_{i_l,i_l+1}\ti \frb\]
to \(\und{\scXr_2}(G_n)\) given by
\[\bar{\pi}_{1,l+2,y,\rho}(X,g_{12},\dots,g_{l+1,l+2},Y)=(X,g_{12}\dots g_{l+1,l+2},Y)\]
and the maps \(\tilde{\pi'}_{k,k+1,y,\bullet}\) are the corresponding twisted linear projections
on \(\und{\scXr_2}(G_2)\).

The maps \(\tilde{\pi'}_{k,k+1,y,\bullet}\) are obtained by a composition
of the restriction of the map \(\tilde{\pi}_{k,k+1,y,\bullet}\) and a
natural projection from \(\und{\scXr_2}(P_{i_k})\) to
\(\und{\scXr_2}(G_2)\). Thus we can apply lemma~\ref{lem:lin} and if we set \(\calF_i=\und{\calcr}_{\eps_i}\) and
set \(\calG=\und{\calcr}_{\parallel}\), we obtain the desired homotopy.
\end{proof}
\section{Link invariant}
\label{sec:link-invariant}

\subsection{Conjugacy class invariant}
\label{sec:conj-class-invar}

First we define conjugacy invariant of the braid. For  that  we observe
the  subvariety \(\frb\ti 1\ti\frb\subset \und{\scXr_2}\) is invariant with respect to the adjoint \(B\)-action.
Thus the corresponding embedding \(j_e:\frb^2\to \und{ \scXr_2}\) induces the functor:
\[j^*_e:\uoMF_{\tau,\sigma}\to \MF_B(\frb^2,\bar{Q}_{\tau,\sigma}), \quad \bar{Q}_{\tau,\sigma}(X,Y)=\Tr(X(\tau\cdot Y-\sigma\cdot Y)).\]

The  potential \(\bar{Q}_{\tau,\sigma}\) is a quadratic function of the diagonal elements, thus there is  a Knorrer periodicity
functor. Let us denote by \(\frb^\alpha\), \(\alpha\in S_n\) fixed by \(\alpha\).
 Then the Knorrer periodicity functor provides us with a functor:
\[\overline{\KN}_{\tau,\sigma}: \MF_B(\frb^2,\bar{Q}_{\tau,\sigma})\to D^{per}_B(\frb^\delta\times \frb^{\delta}),\]
where \(\delta=\tau^{-1}\sigma\).

The  composition of the above functors with \(\CE_{\frn}(\bullet)^T\) and push-forward along
the projection \(\frb^\delta\ti\frb^\delta\to \frh^\delta\ti\frh^\delta\) results into the categorical trace:
\[\mathcal{T}r_0: \uoMF_{\tau,\sigma}\to D^{per}(\frb^\delta\times \frb^{\delta}).\]

There is a natural action of the symmetric group, which changes the labels. This action is the groupoid  homomorphism:
\[\tau\cdot ({}_\alpha\sigma_\beta)=({}_{\alpha\cdot\tau}\sigma_{\beta\cdot\tau}).\]
This automorphism is sent by \(\overline{\und{\Phi^{br}}}\) to the functor induced by the \(\tau\cdot\) action on \(\frb\):
\[\tau: \uoMF_{\alpha,\beta}\to \uoMF_{\alpha\cdot\tau,\beta\cdot\tau}.\]

Now we are ready to state the main property of our trace functor:
\begin{proposition}\label{prop:conj}
  For any \(\calF\in \uoMF_{\alpha,\beta}\), \(\calG\in\uoMF_{\beta,\gamma}\) we have:
  \[\mathcal{T}r_0(\calF\stry\calG)=\mathcal{T}r_0(\tau\cdot \calG\stry\calF),\]
  where \(\tau=\alpha\gamma^{-1}\).
\end{proposition}

It is easier to prove the trace property if we use the convolution in big category. So define a version of the trace functor that works for the convolution
in big category, to be more precise we will work
with the category \(\uMF^\circ_{\sigma,\tau}\). One of the differences between the small and big categories is that in the small category every
object is a module over the ring of function on \(\frh^\delta\ti\frh^{\delta}\) for an appropriate \(\delta\) in the big category
only the structure of \(\CC[\frh^\delta]\)-module is obvious, the action of the other copy of \(\CC[\frh^\delta]\) is hidden.
In construction below we make this action more explicit.

First we define \(\und{\scX_2}^\circ(\delta)=\frg\ti\frb\ti_{\frh}\frb\ti \frh^\delta\) with the \(B\)-adjoint action on all factors.
Similarly to the previous case we define the \(B\)-equivariant map:
\[j_e:\und{\scX_2}^\circ(\delta)\to\und{\scX_2}^\circ, \quad j_e(X,Y_1,Y_2,h)=(X,Y_1,1,Y_2).\]

The  pull-back morphism along \(j_e\) gives us a functor:
\[j_e^*: \uMF^{\circ}_{\tau,\sigma}\to \MF_B(\und{\scX_2}^\circ(\delta),Q_{\tau,\sigma}), \quad Q_{\tau,\sigma}(X,Y_1,Y_2,h)=\Tr(X(\tau\cdot Y_1-\sigma\cdot Y_2)).\]
The potential \(Q_{\tau,\sigma}\) has a quadratic term \(\Tr(X_{--}(Y_{1--}-Y_{2--}))\) thus there is Knorrer periodicity functor:
\[\Phi: \MF_B(\und{\scX_2}^\circ(\delta), Q_{\tau,\sigma})\to \MF_B(\frb\ti\frb\ti\frh^\delta,\bar{Q}_{\tau,\sigma}). \]
The Knorrer functors are intwiners: \(\Phi\circ j_e^*=j_e^*\circ \Phi\). Also since the potential
\(\bar{Q}_{\tau,\sigma}\) is independent of the last factor in \(\frb\ti\frb\ti \ti \frh^\delta\)
we can extend  previously used functor to obtain:
\[\overline{\KN}_{\tau,\sigma}: \MF_B(\frb^2\ti \frh^\delta,\bar{Q}_{\tau,\sigma})\to
  D^{per}_B(\frb^\delta\times \frb^{\delta}\ti \frh^\delta).\]

We denote by \(O(\delta)\) the set of \(\delta\)-orbits  inside \(\{1,\dots,n\}\). There is a natural projection
\(\pi_\delta:\frh\to \frh^\delta\), if \(x_o\), \(o\in O(\delta)\) are coordinates on \(\frh^\delta\) then the projection
is given by \(x_o=\sum_{i\in o} h_i\).
We use same notation for the projection \(\pi_\delta:\frb\to \frh^\delta\) as well
as projection \(\frb^\delta\ti\frb^\delta\ti\frh^\delta \to \frh^\delta\ti \frh^\delta\) which is an identity on the last
factor and \(\pi_\delta\) on the second factor.

Let us fix coordinates \((X,Y_1,Y_2,x)\) on the space  \(\und{\scX_2}^\circ(\delta)\) and introduce Koszul complex
on the space
\[\mathrm{K}^\delta=(R\oti \Lambda^\bullet\langle \theta_1,\dots,\theta_l\rangle,D^\delta),\quad
  D^\delta=\sum_{o\in O(\delta)}(x_o-\sum_{i\in o} X_{ii})\frac{\partial}{\partial \theta_o},\]
where \(l=|O(\delta)|\) and \(R=\CC[ \und{\scX_2}^\circ(\delta)]\).
The  complex \(\mathrm{K}^\delta\) is not \(B\)-equivariant in strong sense but given \(\calC=(M,D,\partial)\in\MF_B(\und{\scX_2}(\delta),Q_{\tau,\sigma})\) by lemma 3.6 of \cite{OblomkovRozansky16}
there is \(\partial': M\oti\Lambda^\bullet\langle \theta_*\rangle\to
M\oti\Lambda^{<\bullet}\langle \theta_*\rangle\) such that
\[(M\oti \Lambda^\bullet \langle \theta_*\rangle, D+D^\delta,\partial+\partial')\in \MF_B(\und{\scX_2}(\delta),Q_{\tau,\sigma}).\]
This matrix factorization is unique up to homotopy (by Lemma 3.7 in \cite{OblomkovRozansky16}) and we use notation
\(\calC\oti_B \mathrm{K}^\delta\) for it.

The trace map is defined by
\[\mathcal{T}r_0:\uMF_{\sigma,\tau}\to D^{per}_B(\frh^\delta\ti\frh^\delta), \quad \mathcal{T}r_0(\calC)=
\pi_{\delta*}\left(\CE_{\frn}(\KN_{\tau,\sigma}\left(\Phi\left( j_e^*(\calC)\otimes_B \mathrm{K}^\delta\right)\right))^T\right).\]
\begin{proof}[Proof of proposition~\ref{prop:conj} ]
  From the construction we see that \(\mathcal{T}r_0(\Phi(\calC))=\mathcal{T}r_0(\calC)\).
  Thus we concentrate on the proof the statement for the trace in the big category.

  If we only care about the structure of \(\CC[\frh^\delta]\)-module (not \(\CC[\frh^\delta\ti \frh^\delta]\)) then the trace property follows from base changes in the commuting diagram:
  \begin{equation}\label{eq:tr-dia}
    \begin{tikzcd}
    \und{\scX_2}^\circ\ti\und{\scX_2}^\circ&\und{\scX_3}^\circ\arrow[l,"j_\Delta"]\arrow[d,"\pi_{13}^\circ"]&
    &\calZ\arrow[ll,dashed,"j"]\arrow[d,dashed,"\pi"]&\\
    &\und{\scX_2}^\circ  &\frg\ti\frb\ti\frb\arrow[l,"j_e"]&\frg\ti \frb^\delta\arrow[l,"j_{KN}"]
    \arrow[r,"\tilde{\pi}_\delta"]&\frh^\delta\ti\frh^\delta
  \end{tikzcd},
\end{equation}
here \(\calZ=\frg\ti G\ti \frb\ti_\frh\frb\) and the new solid arrow maps are:
\[j_\Delta(X,Y_1,g_{12},Y_2,g_{23},Y_3)=(X,Y_1,g_{12}, Y_2)\ti (\Ad_{g_{12}}^{-1}X,Y_2,g_{23},Y_3),\quad
  j_{KN}(X,Y)=(X,Y,Y),\]
\[j_e(X,Y_1,Y_2)=(X,Y_1,1,Y_2),\]
and \(\tilde{\pi}_\delta\) is the natural projection. The dashed arrows are constructed to make the
diagram commute:
\[j(X,g,Y_1,Y_2)=(X,Y_1,g,Y_2,g^{-1},Y_1),\quad \pi(X,g,Y_1,Y_2)=(X,\pi_\delta(Y_1)).\]

If we travel along the solid arrow from the upper-left corner to the lower-right corner we obtain
a simplified version of our trace functor:
\[\CE_{\frn^2}(\tilde{\pi}_{\delta*}\circ j_{KN}^*\circ j_e^*\circ \pi_{13*}^\circ\circ j_\Delta^*(\calF\boxtimes \calG))^{T^2 }= \pi_{\delta*}\left(\CE_{\frn}(\KN_{\tau,\sigma}\left(\Phi\left( j_e^*(\calF\star\calG)\right)\right))^T\right).  \]
On the other hand the base change tell us that
\[ j_{KN}^*\circ j_e^*\circ \pi_{13*}^\circ=\pi_*\circ j^*.\]
The maps \(j_\Delta\circ j\) and \(\tilde{\pi}_\delta\circ\pi\) are equivariant with respect to the involution
\[\mathrm{sw}_z(X,g,Y_1,Y_2)=(\Ad_g^{-1}X,g^{-1},Y_2,Y_1),\]
\[j_\Delta\circ j\circ\mathrm{sw}_z=\mathrm{sw}_x\circ j_\Delta\circ j,
\quad \tilde{\pi}_\delta\circ\pi\circ\mathrm{sw}_z=\tilde{\pi}_\delta\circ\pi,\] where
\(\mathrm{sw}_x\) swaps two copies of \(\und{\scX_2}^\circ\).
Thus the involution \(\mathrm{sw}_z\) induces the isomorphism of complexes of \(\CC[\frb^\delta]\)-modules:
\[\pi_{\delta*}\left(\CE_{\frn}(\KN_{\tau,\sigma}\left(\Phi\left( j_e^*(\calF\star\calG)\right)\right))^T\right)\simeq \pi_{\delta*}\left(\CE_{\frn}(\KN_{\tau,\sigma}\left(\Phi\left( j_e^*(\calG\star\calF)\right)\right))^T\right).\]
To upgrade this isomorphism to the statement about the trace \(\mathcal{T}r_0\) we need to
multiply the varieties of the diagram \eqref{eq:tr-dia} by \(\frh^\delta\) and observe that
\[\mathcal{T}r_0(\calF\star\calG)=\CE_{\frn^2}(\tilde{\pi}_{\delta*}\circ \pi_{*}( j^*\circ j_\Delta^*(\calF\boxtimes \calG)\oti_B\mathrm{K}^\delta))^{T^2 }
  .\]
The only difficulty with applying the previous argument in this case is that the involution
\(\mathrm{sw}_z\) acts non-trivially on the complex \(\mathrm{K}^\delta\). The pull-back of
the complex \(\mathrm{sw}_z^*(\mathrm{K}^\delta)\) has differential:
\[\tilde{D}^{\delta}=\sum_{o\in O(\delta)}(x_o-\sum_{i\in o} \Ad^{-1}_g(X)_{ii})\frac{\partial}{\partial \theta_o}. \]
However, there is a homotopy connecting \(D^\delta\) and \(\tilde{D}^\delta\). If
\(\calF=(M',D',\partial')\) and \(\calG=(M'',D'',\partial'')\) then the partial derivatives
\[u'_o=\sum_{i\in o}\frac{\partial D'}{\partial Y_{ii}}, \quad u''_o=\sum_{i\in o}\frac{\partial D''}{\partial Y_{ii}},\]
provide homotopies between \(\sum_{i\in o}X_{ii}\) and \(\sum_{i\in o}(\Ad_gX)_{ii}\) for \(o\in O(\delta)\).
Thus the matrix \(M=\exp(\sum_{o\in O(\delta)}(u'_o\oti 1+ 1\oti u''_o)\theta)\) 
provides us with the intertwiner:
\[M\circ D^\delta\circ M^{-1}=\tilde{D}^\delta,\]
and the proposition follows.
\end{proof}

\subsection{Framed category and link invariants}
\label{sec:framed-category-link}

In this section we introduce a version of our short category with stability condition. This category has enough structure for
defining the link invariant.

The  framed version of the space \(\und{\scXr_2}\) is a open subset of \(\und{\scXr_2}\ti V\), \(V=\cc^n\) defined by the
stability condition:
\[(X,g,Y,v)\in \und{\scXr_{2,fr}}, \mbox{ iff } \cc\langle X,(\Ad_gY)_+\rangle v=V.\]

The  projection along the vector space \(V\) maps \(\und{\scXr_{2,fr}}\) to the open subset \(\und{\scXr_{2,st}}\subset \und{\scXr_2}\).
The  pull-back along this projection induces the functor:
\[\forg: \uoMF_{\sigma,\tau}\to \MF_{B^2}(\und{\calXr_{2,fr}},\Wr_{\sigma,\tau})=\uoMF_{\sigma,\tau}^{st}.\]

Similarly, we define subspace \(\und{\scXr_{3,fr}}\) as stable locus of the product \(\und{\scXr_3}\ti V\).
There is a natural extension of the maps \(\bar{\pi}_{ij,\rho}\) to the maps between the spaces
\(\und{\scXr_{3,fr}}\) and \(\und{\scXr_{2,fr}}\). Thus we can define the monoidal structure on the category
\(\uoMF_{\bullet,\bullet}^{st}\). Moreover,
the argument of  lemma 12.2 \cite{OblomkovRozansky16} implies that the functor \(\forg\) is monoidal.
\[\forg(\calF\stry\calG)=\forg(\calF)\stry\forg(\calG).\]

The  geometric trace functor from  category \(\uoMF_{\bullet,\bullet}^{st}\) lands into the category constructed on
the space \(\und{\widetilde{\Hilb}}^{free}_{1,n}\) which is the space of triples \((X,Y,v)\in\frb^2\ti V\) satisfying stability
condition:
\[\cc\langle X,Y\rangle v=V.\]
We denote by \(\und{\Hilb}_{1,n}^{free}\) the \(B\)-quotient of \(\und{\widetilde{\Hilb}}_{1,n}^{free}\). We also use notation
\(j_e\) for embedding of \(\und{\widetilde{\Hilb}}^{free}_{1,n}\) inside \(\und{\scX_{2,fr}}\). The  pull-back along this
map give us functor:
\[j_e^*: \uoMF_{\sigma,\tau}^{st}\to \MF_B(\und{\Hilb}_{1,n}^{free},Q_{\sigma,\tau}).\]

There is a natural projection \(\chi:\und{\widetilde{\Hilb}}_{1,n}\to \frh\ti\frh\) and we define
\(\und{\widetilde{\Hilb}}_{1,n}(\calZ):=\chi^{-1}(\calZ)\), \(\calZ\subset \frh\ti\frh\).
We also use abreviation \(\und{\widetilde{\Hilb}}_{1,n}(\delta)\), \(\delta\in S_n\)
for \(\und{\widetilde{\Hilb}}_{1,n}(\frh^\delta\ti\frh^\delta)\).
The Knorrer periodicity thus gives us functor:
\[\KN_{\sigma,\tau}: \MF_B(\und{\Hilb}_{1,n}^{free},Q_{\sigma,\tau})\to D^{per}_B(\und{\widetilde{\Hilb}}_{1,n}(\sigma\tau^{-1})).\]

Combining all the previous constructions we assign a two-periodic complex to a labeled braid \(\beta\in\LBr_n\):
\[\und{\mathbb{S}}_\beta=\CE_{\frn}(\KN_{\sigma,\tau}(\forg(\und{\Phi^{br}}(\beta))))^T,\]
where \(\sigma=s(\beta),\tau=t(\beta)\).
The same argument as above implies that this complex only depends on the conjugacy class of the element of the braid group.

The last component needed to define the knot homology is the  tautological vector bundle \(\calB\) over \(\und{\widetilde{\Hilb}}_{1,n}(\delta)\) with fiber \(V^\vee\).
We use same notation for vector bundle on the quotient \(\und{\Hilb}_{1,n}(\delta)\). The hyper-cohomology
\[\mathbb{H}^k(\beta)=\mathbb{H}(\und{\mathbb{S}}_\beta\oti \Lambda^k\calB)\]
is a doubly-graded module over the doubly-graded ring \(R(\beta)=\cc[X_{11},Y_{11},\dots,X_{nn},Y_{nn}]/I_{\sigma,\tau}\) where
\(I_{\sigma,\tau}\) is generated by the elements \(X_{ii}-X_{\delta(i),\delta(i)}, Y_{ii}-Y_{\delta(i),\delta(i)}\), \(\delta=\sigma\tau^{-1}\), \(i=1,\dots,n\).

Denote \(\ell(\beta)\) the number of connected components of the closure of \(\beta\), which is also equal to the number of orbits of \(\delta\) acting on \(\{1,\ldots,n\}\).
A choice of correspondence between the orbits and elements of \(\{1,\dots,\ell\}\),
identifies \(R(\beta)\) with \(\cc[x_1,y_1,\dots,x_\ell,y_\ell]\). Let us also recall that there is a natural Fourier
transform on \(\mathfrak{F}\) on the ring \(R(\beta)\) given by formula \eqref{eq:Four-ring}.

As in \cite{OblomkovRozansky16} we use notation \(\mathrm{q}^k\mathrm{t}^m\) for the shift of \(qt\)-grading. With this convention in mind
the main result of the paper is the following:
\begin{theorem}\label{thm:main2}
  For any \(\beta\in\LBr_n\) the triply-graded module over \(R(\beta\):
  \[\mathrm{HXY}(\beta)=\oplus_{k\in \zz}\chh^k(\beta),\]
  with \(\chh^k(\beta)\) defined by
  \[\chh^k(\beta)=\mathrm{q}^{\mathrm{wr}(\beta)+n}\cdot \mathbb{H}(\beta),\]
  is an isotopy invariant of the closure \(L(\beta)\), here \(\mathrm{wr}(\beta)\) is the wreath of the closure
  of \(\beta\).

  Moreover, there is an  involution \(\mathfrak{F}\) on \(\HY(\beta)\) that intertwines the
  the Fourier transform on \(R(\beta)\).
\end{theorem}

The  theorem follows  from more general statement. Using the projection \(\chi\)
we define:
\[\mathcal{T}r_k(\beta)=\chi_*(\und{\mathbb{S}}_\beta\oti\Lambda^k\calB)\in D^{per}(\frh^\delta\ti\frh^\delta), \quad
  \dot{\beta}=\delta.\]
Respectively, we denote by \(\mathcal{T}r\) the direct sum
\[\mathcal{T}r(\beta)=\oplus_k\mathcal{T}r_k(\beta).\]
Previously defined trace \(\mathcal{T}r_0\) could restricted to the stable part of the corresponding spaces to define:
\[\mathcal{T}r_0: \uoMF_{\tau,\sigma}^{st}\to D^{per}(\frh^\delta\ti\frh^\delta),\quad \delta=\tau^{-1}\sigma.\]
This definition is consistent with the other construction of the trace since
\begin{equation}\label{eq:tr-tr}
  \mathcal{T}r_0(\beta)=\mathcal{T}r_0(\und{\calC}_\beta).
\end{equation}

The  first part of the theorem~\ref{thm:main2} follows from the theorem bellow the second part is discussed in the next section.
\begin{theorem}
  For any \(\beta\in \Br_n\) the direct sum
  \[\mathcal{E}(\beta)=\mathrm{q}^{\mathrm{wr}(\beta)+n}\cdot \mathcal{T}r(\beta)\]
  is a link invariant of the closure \(L(\beta)\).
  \end{theorem}
\begin{proof}
  We need to check that the trace satisfies the Markov move conditions. The first Markov move is about
  conjugacy invariance and it follows immediately from \eqref{eq:tr-tr} and \ref{prop:conj}.
The second Markov move is equivalent to the pair of equations
  \begin{equation}
    \label{eq:Mark2}
    \mathcal{T}r_k(\beta\cdot \sigma(1))=\mathcal{T}r_k(\beta),\quad \mathcal{T}r_k(\beta\cdot\sigma(1))=\mathcal{T}r_{k-1}(\beta),
  \end{equation}
  where \(\beta\in\Br_{n-1}\) is a braid on the strands \(2,\dots,n\).

  These equations are implied by the same argument as theorem 13.3 of \cite{OblomkovRozansky16}. Below we outline the
  steps of the proof. The reader may consult \cite{Oblomkov18a} for a detailed account.

  The  images of the braids \(\beta\) and \(\beta\cdot \sigma(1)\)  in the symmetric groups are \(\delta'\in S_{n-1}\) and
  \(\delta=\delta'\cdot t_{12}\in S_{n}\).
  Now we can use the nested nature of the scheme \(\Hilb_{1,n}\) to define the projection map:
\[\pi: \Hilb^{free}_{1,n}\to \cc_x\ti\cc_y\times \Hilb^{free}_{1,n-1},\]
where the first two components of the map \(\pi\) are \(x_{11},y_{11}\) and the last component is just forgetting of the first
row and the first column of the matrices \(X,Y\) and the first component of the vector \(v\). Let us also fix notation for the line bundles on
 \(\Hilb_n^{free}\): we denote by \(\calO_k(-1)\) the line bundle induced from the twisted trivial bundle \(\calO\otimes\chi_k\).
It is
shown in proposition 13.1 of \cite{OblomkovRozansky16} that  the fibers of the map \(\pi\) are projective spaces \(\mathbb{P}^{n-1}\) and
  \begin{equation}
  \calB_n/\pi^*(\calB_{n-1})=\calO_n(-1),
  \quad \calO_n(-1)|_{\pi^{-1}(z)}=\calO_{\mathbb{P}^{n-1}}(-1).
  \end{equation}

We can combine the last proposition with the observation that the total homology \(H^*(\mathbb{P}^{n-1},\calO(-l))\) vanish
if \(l\in (1,n-1)\) and is one-dimensional for \(l=0,n\) and obtain a conclusion of proposition 13.2 of \cite{OblomkovRozansky16}
  For any \(n\) we have:
  \begin{itemize}
  \item \(\pi_*(\Lambda^k\calB_n)=\Lambda^k\calB_{n-1}\)
  \item \(\pi_*(\calO_n(m)\oti \Lambda^k\calB_n)=0\) if \(m\in [-n+2,-1]\).
    \item \(\pi_*(\calO_n(-n+1)\oti \Lambda^k\calB_n)=\Lambda^{k-1}\calB_{n-1}[n]\)
  \end{itemize}



  The main technical component of the proof is a careful analysis of  matrix factorizations \(\bar{\calC}_{\beta\cdot \sigma(1)^{\pm 1}}\in\MF(\bar{\scX}_n^{st},\Wr)\), see proof of theorem 13.3 in \cite{OblomkovRozansky16}. The same argument is applicable for
  study of the curved complexes \(\und{\calcr}_{\beta\cdot \sigma(1)^{\pm 1}}\in\uoMF^{st}_{\tau,\sigma}\).
  To state the version of the result we need to work with the space \(\Hilb_{1,n}(\frh\ti \frh^{\sigma(1)})\),
  we denote the embedding map into \(\Hilb_{1,n}\) by \(i_\Delta\). The image \(\pi(\Hilb_{1,n}(\frh\ti \frh^{\sigma(1)}))\)
  is isomorphic to \(\cc_x\ti \Hilb_{1,n-1}\). Now the version of the result from \cite{OblomkovRozansky16} that we need
 show  that curved complex \(i^*_\Delta(\und{\calcr}_{\beta\cdot\sigma(1)^{\epsilon}})\) has form:
   \begin{equation}\label{dia:big}
     \begin{tikzcd}
              \calC'\ar[r,bend right]\ar[rrr,bend right]\ar[rrrrr,bend right]&
       \calC'\ot V\ar[l,dotted]\ar[r,bend right]\ar[rrr,bend right]&
       \calC'\ot\Lambda^2V\ar[l,dotted]\ar[r,bend right]\ar[rrr,bend right]&
       \calC'\ot\Lambda^3V\ar[l,dotted]\ar[r,bend right]&
      \calC'\ot\Lambda^4V\ar[l,dotted]\ar[r,bend right]&\cdots\ar[l,dotted]
  \end{tikzcd}
\end{equation}
  where \(\calC'=\pi^*(\und{\calcr}_{\beta}\oti \mathrm{K}[x_{11}])\), \(V=\CC^{n-2}\),
 the dotted arrows are the differentials of the Koszul complex for the ideal \(I=(g_{13},\dots,g_{1n})\) where \(g_{ij}\) are the coordinates on the
  group inside the product \(\und{\scXr}_n=\frb_n\ti \GL_n\ti \frb_n\). Thus after the pull-back \(j_e^*\) the dotted arrows of the curved complex
  vanish and we only left with the arrows going from the left to right.

  Since the push-forward along \(\pi\) intertwines the Knorrer functor:
  \[\CE_{\frn_n}(\KN_{\tau,\sigma}(\calc))=\CE_{\frn_{n-1}}(\KN_{\tau',\sigma'}(\pi_*\circ i^*_\Delta(\calc)))^{T_{n-1}}.\]
  to finish proof of the theorem we need to compute
   \(\pi_*(i_\Delta^*\circ j^*_e(\und{\calcr}_{\beta\cdot\sigma_1^{\pm 1}})\ot \Lambda^k\calB_n)\) and  here we can apply the previous corollary.
  Thus if \(\epsilon=1\) then only the left extreme term of \(j_e^*\) of the complex \eqref{dia:big} survive the push-forward \(\pi_*\). Since
  the non-trivial arrows of \(j_e^*\) of \eqref{dia:big} all  are the solid arrows which are  going the left to the right, the contraction of the \(\pi_*\)-acyclic terms
  do not lead to appearance of new correction arrows thus conclude that
  \[\pi_*(i^*_\Delta\circ j^*_e(\und{\calcr}_{\beta\cdot\sigma_1})\ot \Lambda^k\calB_n)=j_e^*(\und{\calcr}_\beta\ot \Lambda^k\calB_{n-1}).\]

  If \(\epsilon=-1\) then only the right extreme term of \(j_e^*\) of the complex \eqref{dia:big} survive the push-forward \(\pi_*\). Hence the
  similar argument as before implies:
    \[\pi_*(i^*_\Delta\circ j^*_e(\und{\calcr}_{\beta\cdot\sigma_1^{-1}})\ot \Lambda^k\calB_n)=j_e^*(\und{\calcr}_\beta\ot \Lambda^{k-1}\calB_{n-1}).\]

\end{proof}

\subsection{Fourier transform for the trace}
\label{sec:four-transf-trace}

Let us use notation for the anti-involution inverting the direction of the braid:
\[\overline{\alpha\cdot\beta}=\beta\cdot \alpha,\quad \overline{\sigma(i)}=\sigma(i)\]

Let us also define an involution on the ring \(R(\beta)\) by
\[\FT(x_i)=y_i,\quad \FT(y_i)=x_i.\]
\begin{theorem}\label{thm:FTd}
  For any \(\beta\in \LBr_n\) we  have
  \[\FT(\mathcal{T}r(\beta))=\mathcal{T}r(\overline{\beta})\]
\end{theorem}
\begin{proof}
  The statement follows from the theorem~\ref{thm:lin} and the fact the properties of \(\und{\calcr}_+\).
  Indeed, \(\und{\calcr}_+\in \MF_{\tau,1}\) is the Koszul matrix factorization for the pair \(\tilde{x_0}\)
  and \(\tilde{y}_0\). The Fourier transform \(\FT\) switches \(\tilde{x_0}\)
  and \(\tilde{y}_0\) hence:
  \[\FT(\und{\calcr}_+)=\und{\calcr}_+\]
  and the statement follows.
\end{proof}

Next we discuss the easy involution on our categories that manifests itself in the equation \eqref{eq:reverse}.
We define an involution on the big space \(\und{\scX_2}\) by:
\[\mathfrak{I}(X,g_1,Y_1,g_2,Y_2)=(X,g_2,-Y_2,g_1,-Y_1).\]

This involution induces the isomorphism of categories:
\begin{proposition}\label{prop:I}
  The involution \(\mathfrak{I}\) induces the isomerism's of categories:
  \[\mathfrak{I}: \und{\MF}_{\sigma,\tau}\to \und{\MF}_{\tau,\sigma}\]
  such that
  \[\mathfrak{I}(\calF\star\calG)=\mathfrak{I}(\calG)\star\mathfrak{I}(\calF).\]
\end{proposition}

The  corresponding involution on the ring \(R(\beta)\) is given by:
\[\mathfrak{I}(x_i)=x_i,\quad \mathfrak{I}(y_i)=-y_i.\]
The Fourier transform \(\FTT\) is the composition of these two involutions:
\[\FTT=\FT\circ \mathfrak{I}.\]

Thus combining theorem \ref{thm:FTd} and proposition \ref{prop:I} we derive that for any
\(\beta\in\LBr_n\) we have:
\[\FTT(\mathcal{T}r(\beta))=\mathcal{T}r(\beta),\quad \HY(L(\beta))=\mathfrak{F}(\HY(L(\beta)))\]
and Theorems \ref{thm:symm}, \ref{thm:main2} and \ref{thm:obj} follow.

\subsection{Relation with HOMFLYPT homology}
\label{sec:relat-with-homflypt}

In our previous papers we worked with homology theory \(\mathrm{H}(\beta)\) which categorifies the HOMFLYPT
polynomial. In this section we explain how  the old homology theory can be obtained from the new
homology theory \(\HY(\beta)\).

Let us denote by \(R(\beta)_x\) the ring \(\cc[x_1,\dots,x_\ell]\) where \(\ell\) is number of connected components
of the closure \(L(\beta)\). The ring \(R(\beta)_x\) is a \(R(\beta)\)-module and in the next section we show
\begin{theorem}\label{thm:Hb}
  For any \(\beta\in \LBr_n\) we have the relation between the homology of
  the closure of the braid:
  \[\HY(\beta)\Ltimes_{R(\beta)}R(\beta)_x=\mathrm{H}(\beta).\]
\end{theorem}

The  relation between the homology theories is especially simple in the case when \(L(\beta)\) is a knot, as it is
stated in theorem~\ref{thm:free}.
\begin{proof}[Proof of theorem~\ref{thm:free}]
  From the construction of the trace \(\mathcal{T}r(\beta)\), \(\beta\Br_n\) \(K=L(\beta)\) we see that the variables \(x,y\) in
  \(R(K)=\cc[x,y]\) are the sums of the diagonal elements \(x=\sum_{i=1}^n x_{ii}\), \(y=\sum_{i=1}^n y_{ii}\). Since these elements
  are not changed by a conjugation, the statement follows.
\end{proof}

Thus we have shown \eqref{eq:fd-hom} and the theorem~\ref{thm:main} follows.

\section{Two-colored  link invariants}
\label{sec:relation-with-other}

It is very likely that that all known triply-graded categorifications of the knot homology coincide or differ by
some trivial factor. However the case of links is more interesting, it is known that there are several link
invariants that are equal HOMFLYPT polynomial on the knots but differ from HOMFLYPT on the general links.
Thus we do not expect any kind of uniqueness of HOMFLYPT homology of links.

In this section we discuss the invariants of the links colored with two colors.  If link is a knot the
homology discussed here match with homology from the previous sections and homology from \cite{OblomkovRozansky16}.
For multi-component links we expect existence of the spectral sequence connecting these homology
with the dualizable link homology.

\subsection{Invariants of dichromatic links}
\label{sec:invar-dichr-links}

Since we will work with the links colored with two colors, we are lead to a definition of the groupoid
of dichromatic braids \(\Br_n^{\dic}\). An element of the groupoid is a triple \(s,\beta,t\), \(s,t\subset \{1,\dots,n\}\),
\(\beta\in \Br_n\) and \(\dot{\beta}(s)=t\). It is convenient to use notation \(\und{\beta}=(s,\beta,t)\)
and \(s(\und{\beta})=s, t(\und{\beta})=t\). The braids \(\und{\alpha}\cdot\und{\beta}\) is defined if \(t(\und{\alpha})=
s(\und{\beta})\).

Now we would like to construct an analogue of the homomorphisms \(\Phi^{\lbr}\) and \(\bar{\Phi}^{\lbr}\) for the dichromatic braids.
We define \(\frb[S]\), \(S\subset \{1,\dots,n\}\) to be a subspace of the algebra \(\frb\) defined by vanishing of \(i\)-th diagonal
entry if \(i\notin S\).
For given \(c,c'\subset \{1,\dots,n\}\) that we introduce  space spaces:
\[\scX_2[c,c']=\frg\ti G\ti \frb[c]\ti G\ti \frb[c'],\quad \overline{\scX}_2[c,c]=\frb[c]\ti G\ti\frb[c'].\]
 Respectively, we define the corresponding categories of matrix factorizations:
\[\MF[\sigma,c,\tau]=\MF_{G\ti B^2}(\scX_2[\sigma^{-1}(c),\tau^{-1}(c)],W_{\sigma,\tau}),\quad
  \overline{\MF}[\sigma,c,\tau]=\MF_{B^2}(\scXr_2[\bar{c},\tau^{-1}(c)],\Wr_{\sigma,\tau}).\]
where  \(W_{\sigma,\tau}\) \(\Wr_{\sigma,\tau}\) are given by the formulas \eqref{eq:pot1} and \eqref{eq:pot2}.

We have the relation between the potentials:
\begin{multline*}
W^\circ_{\sigma,\tau}(X,g,Y_1,Y_2)=\Wr(X',g,Y'_2)+\sum_{i\in  c}X_{ii}(Y_{1,\sigma(i),\sigma(i)}-\Ad_g(Y_2)_{\tau(i),\tau(i)})\\+\Tr(X_{--}((Y_1)_{++}+\Ad_g(Y_2)_{++})),\end{multline*}
where \(X',Y'_2\) are projections of \(X,Y'\) on \(\frb[\bar{c}]\) and
\(\frb[\tau(c)]\) respectively. Thus
just like in previous cases there is a Knorrer functors:
\[\begin{tikzcd}
\overline{\MF}[\sigma,c,\tau]\arrow[rr,shift left=1,"\Phi"]&& \MF[\sigma,c,\tau]\arrow[ll,shift left=1,"\Psi"]
\end{tikzcd}
\]
such that \(\Psi\circ\Phi=1\). We call the category \(\overline{\MF}[\sigma,c,\tau]\) a {\it short version} of
the category \(\MF[\sigma,c,\tau]\).

The  convolution space for the big category  in this setting is \(\scX_3[c]=\frg\ti (G\ti\frb[c])^3\) and
the  monoidal structure \(\star\) on the categories \(\MF[\bullet,\bullet,\bullet]\)
 is defined by the formulas
 \eqref{eq:big-con}:
 \[\star: \MF[\sigma,c,\tau]\ti \MF[\tau,c,\rho]\to \MF[\sigma,c,\rho].\]

 We can define the convolution structure  \(\bar{\star}\) on the category \(\overline{\MF}[\bullet,\bullet,\bullet]\)
 using Knorrer functors but we can also define the convolution in short category explicitly. To define the convolution
 few notations are needed. For  a subset \(c\subset \{1,\dots,n\}\) we generalize the upper-triangular truncation by
 \[X_{+c}=X_{++}+\Delta[c](X),\]
 where \(\Delta[c](X)\) is the diagonal matrix consisting of \(i\)-th, \(i\in c\) diagonal entries of \(c\).

 The modified projection maps \(\overline{\scX_3}[\rho,c]=\frb[\bar{c}]\ti G^2\ti\frb[\rho^{-1}(c)]\)
 to \(\overline{\scX_2}[\bar{c},\sigma^{-1}(c)]\), \(\overline{\scX_2}[\bar{c},\rho^{-1}(c)]\) and \(\overline{\scX_2}[\bar{c},\rho(c)]\)
 \[\bar{\pi}_{12}[\sigma,\rho,c](X,g_{12},g_{23},Y)=(X,g_{12},\Ad_{g_{23}}(\rho\cdot Y)_{\sigma(c)+}),\]
 \[\bar{\pi}_{23}[\sigma,\rho,c](X,g_{12},g_{23},Y)=(\Ad_{g_{12}}^{-1}(X)_{\bar{c}+},g_{23},Y),\]
 \[\bar{\pi}_{13}[\sigma,\rho,c](X,g_{12},g_{23},Y)=(X,g_{12}g_{23},Y).\]

 Thus we define the convolution in the short category by
 \[\calF\bar{\star}\calG=\bar{\pi}_{13*}[\sigma,\rho,c](\CE_{\frn^{(2)}}(\bar{\pi}_{12}^*[\sigma,\rho,c](\calF)\otimes_B \bar{\pi}_{23}^*[\sigma,\rho,c](\calG))^{T^{(2)}})\in\overline{\MF}[\sigma,c,\rho],\]
 for \(\calF\in \overline{\MF}[\sigma,c,\tau],\calG\in\overline{\MF}[\tau,c,\rho].\)

Many of the categories \(\overline{\MF}[\bullet,\bullet,\bullet]\) are isomorphic
 since for any \(\alpha\in S_n\) there is a natural functor:
 \begin{equation}\label{eq:shift}
   \alpha: \overline{\MF}[\sigma,c,\tau]\to \overline{\MF}[\sigma\alpha,\alpha^{-1}(c),\tau\alpha].\end{equation}


 The  the subset \(\Br_{n,k}^{\dic}\subset \Br_n^{\dic}\) of dichromatic braids \(\und{\beta}\) with \(|s(\und{\beta})|=k\) is
 a sub-algebroid. For any \(c\), \(|c|=k\) there is natural homomorphism the algebroids:
 \[\pi_c: \LBr_k\to \Br^{\dic}_{n,k},\]
 defined by \(\pi_c((s,\beta,t))=(s(c),\beta, t(c))\). Below we discuss the counter part of
 this homomorphism in convolution algebras of matrix factorizations.

 The inclusion map \(i[c,c']:\scXr_2[c,c']\to \und{\scXr_2}\) provides
 us the convolution algebra homomorphism. We will not need full
 generality of the intertwining statement, for our purposes it is
 enough to show the analog of lemma~\ref{lem:lin} in this setting.

 \begin{proposition}\label{prop:lbr}
   We have the commuting square of homomorphisms:
   \[
     \begin{tikzcd}
       \LBr_n\arrow[d,"\pi_c"]\arrow[r,"\Phi"]&(\uoMF_{\bullet,\bullet},\stry)\arrow[d,"i^*"] \\
       \Br^{\dic}_{n,k}\arrow[r,"\Phi"]&( \overline{\MF}[\bullet,c,\bullet],\star)
     \end{tikzcd},
   \]
   where \(i\) is an abbreviation for \(i[\bar{c},\bullet]\).
 \end{proposition}
 \begin{proof}
   We only need to show that \(i^*\) is a homomorphism of the convolution algebras, the bottom arrow of the
   square is defined by commutativity of the square.  We denote by \(i[\sigma,c]\)
 the natural embedding of \(\scXr_3[\rho,c]\) into \(\und{\scXr_3}\) then lemma below and the argument of the proof of \ref{thm:lin}  implies the statement.
  \end{proof}

  \begin{lemma}\label{lem:lin1}
    Let us assume that \(\calF=(M',D',\partial')\in \uoMF_{\sigma,\tau}\),
    \(\calG=(M'',D'',\partial'')\in \uoMF_{\tau,\rho}\) and that the differentials
    \(D',\partial',D'',\partial''\) are linear along the factors \(\frb\) in
    \(\und{\scXr_2}\). Then there is a  homotopy
    \[i^*[\rho,c](\bar{\pi}_{12,y,\rho}^*(\calF)
   \oti_B\bar{\pi}_{23,y,\rho}^*(\calG))\sim \bar{\pi}_{12}^*[\sigma,\rho,c](i^*[\bar{c},\sigma^{-1}(c)](\calF))
   \oti_B\bar{\pi}_{23}^*[\sigma,\rho,c](i^*[\bar{c},\rho^{-1}(c)](\calG)).\]
\end{lemma}
\begin{proof}
  We can write an explicit formula for the homotopy. To simplify notations, we
  assume that all equivariant correction differentials vanish. Also we  set \(\tau=1\),
  because of  \eqref{eq:shift} this assumption is not restrictive.

  We also identify \(\calF\), \(\calG\) with the pull-backs
  \(  \bar{\pi}_{12,y,\rho}^*(\calF)\),\(\bar{\pi}_{23,y,\rho}^*(\calG)\) and we use notation \(\tilde{\calF}\) and
  \(\tilde{\calG}\) for the pull-backs \(\bar{\pi}_{12}^*[\sigma,\rho,c](i^*[\bar{c},\sigma^{-1}(c)](\calF))\) and \(\bar{\pi}_{23}^*[\sigma,\rho,c](i^*[\bar{c},\rho^{-1}(c)](\calG)).\)
    The convolution spaces \(\und{\scXr_3}\) and \(\scXr_3[\rho,c]\) have  standard coordinate
  systems \(X,g_{12},g_{23},Y\) and two other coordinate systems pulled back along the
  maps \(\bar{\pi}_{12,y,\rho}\), \(\bar{\pi}_{23,y,\rho}\)
  and \(\bar{\pi}_{12}[\sigma,\rho,c]\), \(\bar{\pi}_{23}[\sigma,\rho,c]\).
  We denote the last two coordinate system by \((X',g_{12},g_{23},Y')\) and \(X'',g_{12},g_{23},Y''\).

  By taking the derivatives
  of the potentials \(\Wr_{\sigma,1}\),
   \(\Wr_{1,\rho}\) we get:
   \begin{equation}\label{eq:diffs2}
   [\partial_{Y'_{ii}}D',D']=\sigma^{-1}\cdot X'_{ii}-(\Ad_{g_{12}}^{-1}X')_{ii},\quad
     [\partial_{X''_{ii}}D'',D'']=Y''_{ii}-(\Ad_{g_{23}}\rho\cdot Y'')_{ii}.\end{equation}
   Since all differentials and the potentials are \(\frb\)-linear by taking the second
   derivatives of the potentials along \(\frb\) we prove that operators
   \[u_i=\partial_{Y'_{ii}}D'\oti\partial_{X''_{ii}}D'',\]
   satisfy relations \(u_i^2=0\), \(u_iu_j+u_ju_i=0.\)
   Hence we can define the homotopy map by
   \[U=\exp(-\sum_{i\in\sigma(c)}^n u_i).\]
   Finally, let us observe that in  \eqref{eq:diffs2} \(\sigma^{-1}\cdot X'_{ii}=0\) for \(i\in\sigma(c)\) and
   \((\Ad_{g_{12}}^{-1}X')_{ii}\), \(i\in \sigma(c)\) is sent to zero by the pull-back \(i^*[\rho,c]\).
  Thus   combining all previous observations we arrive to
   \[U(D'\otimes I+I\oti D'') U^{-1}=\tilde{D'}\oti I+I\oti\tilde{D''},\]
   where \(\tilde{D'}\) and \(\tilde{D''}\) are the differentials of the pull-backs
   \(  \tilde{\calF}\),\(\tilde{\calG}\).
\end{proof}






 \subsection{Traces}
\label{sec:traces}

There is natural trace functor on the categories \(\overline{\MF}[\sigma,c,\tau]\). Similar to the previous construction
we define the embedding and the pull-back:
\[j_e[c',c'']:\frb[c']\ti \frb[c'']\to \scXr_2[c',c''],\quad j_e^*[c',c'']: \MF[\sigma,c,\tau]\to D^{per}_B(\frb[\bar{c}]\ti \frb[\tau^{-1}(c)]).\]
The set of diagonal elements of  the product  \( \frb[\bar{c}]\ti \frb[\tau^{-1}(c)]\) naturally isomorphic
to \(\frh\). Thus by taking the \(B\)-invariants and tensoring with the exterior powers of the  tautological bundle we arrive to
the functor:
\[\mathcal{T}r'_k:\MF[\sigma,c,\tau]\to D^{per}(\frh), \quad \mathcal{T}r'_k(\calF)=\CE_{\frn}(j_e^*(\calF)\oti \Lambda^k\calB)^T.\]

The  dichromatic braid \(\und{\beta}\) is called closable if \(s(\und{\beta})=t(\und{\beta})\). For  a closable braid \(\und{\beta}\) both sets
\(s(\beta)\) and \(t(\beta)\) are unions of the orbits of the action of \(\und{\dot{\beta}}\). Let us assume that \(\und{\beta}=\pi_c({}_\tau\beta_\sigma)\) then
 the identification between \(\frh\)
 and \(\frh[\bar{c}]\ti\frh[\tau^{-1}(c)]\) naturally extends to the identification between corresponding \(\dot{\beta}\)-invariant pieces of the spaces.
 We denote by \(\pi[\sigma,c,\tau]\) the corresponding  linear projection \(\frh^\delta\to \frh\) and for the triples \(\sigma,c,\tau\) such that \(\sigma(c)=\tau(c)\) we
 define:
 \[\mathcal{T}r_k:\MF[\sigma,c,\tau]\to D^{per}(\frh^\delta), \quad \mathcal{T}r_k=\pi_*[\sigma,c,\tau]\circ\mathcal{T}r'_k.\]
 where \(\delta=\tau^{-1}\sigma\).

 This last trace is compatible with the previously defined trace via pull-back morphism. Let \(\und{j}[\sigma,c,\tau]\)
 be the embedding of \(\frh^\delta=\frh[\bar{c}]^\delta\ti\frb[\tau(c)]^\delta\) inside \(\frh^\delta\ti \frh^\delta\), then we have
 \begin{proposition}
   For \(\calF\in \uoMF_{\sigma,\tau}\) and \(c\) such that \(\sigma(c)=\tau(c)\) we have the following relation between the traces:
   \[\und{j}^*[\sigma,c,\tau]\circ\mathcal{T}r_k(\calF)=\mathcal{T}r_k(i^*[\bar{c},\tau^{-1}(c)](\calF)).\]
 \end{proposition}
 \begin{proof}
   We claim that both sides of the equations are two ways to travel from lower right  corner of
   the commuting diagram to the upper left corner:
   \[
     \begin{tikzcd}
       \frh^\delta\arrow{d}{\und{j}[\sigma,c,\tau]}&\arrow{l}{\pi[\sigma,c,\tau]}\frh[\bar{c}]\ti\frh[\tau^{-1}(c)]\arrow{r}\arrow[d,dashed]&\frb[\bar{c}]\ti \frb[\tau^{-1}(c)]\arrow{rr}{j_e[\bar{c},\tau^{-1}(c)]}&&\scXr_2[\bar{c},\tau^{-1}(c)]\arrow{d}{i[\bar{c},\tau^{-1}(c)]}\\
       \frh^\delta\ti\frh^\delta&\frh^\delta\ti\frh\arrow{r}{j_{KN}}\arrow{l}{\pi_\delta}&\frh\ti\frh\arrow{r}&\frb\ti\frb\arrow{r}{j_e}&\und{\scXr_2}
            \end{tikzcd}.
   \]
To be more precise to get both sides of the equation we need to pull-back along the labeled
 arrows, with exception of the arrows with labels  \(\pi[\sigma,c,\tau]\) and \(\pi_\delta\), where we use push-forward, and apply the averaging functor \(\CE_\frn(\bullet)^T\) to in the reverse direction along the labels
 unmarked arrows.

 Maps \(j_{KN}\) is defined as an embedding of \(\frh^\delta\ti \frh\) into \(\frh\ti\frh\) defined with the help of identification
 between \(\frh[\bar{c}]^\delta\ti \frh[\tau^{-1}(c)]^\delta\) and \(\frh^\delta\) and \(\frh[c]\ti \frh[\tau^{-1}(\bar{c})]\).  We use the same
 identifications to define the projection \(\pi_\delta\). The  dotted arrow is given by the projection of
 \(\frh=\frh[\bar{c}]\ti\frh[\tau^{-1}(c)]\) onto \(\frh^\delta=\frh^\delta\ti 0\subset \frh^\delta\ti\frh\).

 The Knorrer functor from the definition of the trace is the composition
 of the functors \(\pi_{\delta*}\circ j_{KN}^*\). Thus indeed the RHS of the equation is obtained by tracing the
 diagram the way we described.
The fact that the LHS is obtain this way is immediate from the construction of the trace.

Thus to prove the equation we use the base change for the left commuting square of the diagram and commuting property
of right commuting square of the diagram.
 \end{proof}

\subsection{Relations between the invariants}
\label{sec:relat-betw-invar}

Using \eqref{eq:shift} we can define the monoidal structure on the union of categories
\(\overline{\MF}[\sigma,c]:=\overline{\MF}[\sigma,c,1]\) by
\[\calF\tilde{\star}\calG:=\calF\stry (\tau^{-1}\cdot \calG),\quad \calF\in\overline{\MF}[\sigma,c],\quad\calG\in \overline{\MF}[\tau,c'],\]
here \(c=\tau(c')\). Let us also point out that
\[\overline{\MF}[\sigma,c]=\MF_B(\frb[\bar{c}]\ti G\ti \frb[c],W)\]
does not depend on \(\sigma\) but this index participate in the definition of the monoidal structure.
As usual we use notation \(\overline{\MF}[\bullet,\bullet]\) for the union of these categories.

The  proposition~\ref{prop:lbr} implies that we have
\begin{corollary}
  There is a groupoid homomorphism:
  \[\Phi^{\dic}: \Br_n^{\dic}\to \MF[\bullet,\bullet]\]
  such that if \(\pi_c({}_1\beta_s)=\und{\beta}\) then
  \[\Phi^{\dic}(\und{\beta})=i^*[\bar{c},s(c)]\circ \Phi^{\lbr}({}_1\beta_s).\]
\end{corollary}

Using homomorphism from the corollary we define the trace on \(\Br_n^{\dic}\) by
\[\mathcal{T}r:\Br_{n;\ell}^{\dic}\to D^{per}(\cc^{\ell}),\quad \mathcal{T}r(\und{\beta}):=
  \mathcal{T}r(\Phi^{\lbr}({}_1\beta_s)),\]
where \(\pi_c({}_1\beta_s)=\und{\beta}\) and \(\Br_{n;k,\ell}^{\dic}\) is the set of braids
\(\und{\beta}\) with \(\ell\) connected components. The  trace is defined only for closable braids.

For  a closable braid \(\und{\beta}\in \Br_{n;\ell}\) we define the ring \(R(\und{\beta})=\cc[x_1,\dots,x_k,y_1,\dots,y_{\ell-k}]\)
where \(k\) is the number of \(\und{\dot{\beta}}\)-orbits on the set \(s(\und{\beta})\). The ring
\(R(\und{\beta})\) is a module over the ring \(R(\beta)=\cc[x_1,\dots,x_\ell,y_1,\dots,y_\ell]\) and
can summarize the results of this section in

\begin{theorem}
  For any closable \(\und{\beta}\in \Br_n^{\dic}\) the element
  \[\mathcal{T}r(\und{\beta})\in D^{per}(\Spec(R(\und{\beta})))\]
  is an isotopy invariant of the closure and
 \[\mathrm{T}r(\und{\beta})=\mathrm{T}r(\beta)\Ltimes_{R(\beta)}R(\und{\beta}).\]
\end{theorem}

The groupoid \(\Br_{n,0}^{\dic}\) consisting of the braid \(\und{\beta}\) with \(s(\und{\beta})=\emptyset\) is
equal to the braid group and the module of derived global sections of the trace is exactly the homology
from \cite{OblomkovRozansky16}:
\[R\Gamma(\mathcal{T}r(\und{\beta}))=\mathrm{H}(\beta).\]

Hence the theorem \ref{thm:Hb} follows.

 \section{Further directions and conjectures}
\label{sec:conjs}

\subsection{Relation with y-fied homology}
\label{sec:relation-with-y}

In the paper \cite{GorskyHogancamp17} the authors construct an isotopy invariant of the closure of a braid
\(L(\beta)\), they use notation \(\mathrm{HY}(\beta)\) for the invariant.
The  invariant \(\mathrm{HY}(\beta)\) has many features of the homology \(\HY(\beta)\). In particular
\(\mathrm{HY}(\beta)\) is naturally a module over the ring
\(\cc[x_1,\dots,x_\ell,y_1,\dots,y_\ell]\) where \(\ell\) is the number of connected components of the
closure.

The  authors of \cite{GorskyHogancamp17} also conjecture that their invariant is symmetric with respect
to the switch of \(x\) and \(y\)-variables. So it is natural to put forward

\begin{conjecture}
  \[\mathrm{HY}(\beta)=\HY(\beta).\]
\end{conjecture}

\subsection{Flag Hilbert schemes and pure braids.}
\label{sec:flag-hilbert-schemes}

The  pure braid group \(\PBr_n\) embeds inside \(\LBr_n\) as subset of elements \(\beta\) with
the property \(s(\beta)=t(\beta)=1\). In particular, we have homomorphism:
\[\Phi^{pbr}:\PBr_n\to \uoMF_{1,1}.\]

The  category \(\uoMF_{1,1}\) is very close to the category dg category \(\mathrm{DG}(FHilb_n)\) of the flag Hilbert scheme.
Indeed, in this case the potential is equal:
\[\Wr_{1,1}(X,g,Y)=\Tr(X(Y-\Ad_gY)).\]
In particular, if \(g\) is in the first formal neighborhood of \(1\) then the potential becomes
a standard potential for dg structure of \(FHilb\). This observation together with some explicit computations
from \cite{OblomkovRozansky18} motivate us to propose
\begin{conjecture}
  There is an \(A_\infty\)-deformation \((\mathcal{A}_\infty(FHilb_n),\tilde{\otimes})\) of
  \((\mathrm{DG}(FHilb_n),\otimes)\) such that
  \[(\mathcal{A}_\infty(FHilb_n),\tilde{\otimes})\simeq (\uoMF_{1,1},\stry).\]
\end{conjecture}

\subsection{Quantization}
\label{sec:quantization}

The ring \(R(\beta)\) has a natural symplectic structure \(\omega\). In particular,
\(R(\beta)\) has a natural quantization \(D(\beta)\) as ring of differential operators on \(\cc^\ell\).
The invariants of two-colored links as we explain is a restriction of the invariant
\(\mathcal{E}(L(\beta))\) to the Lagrangian subvariety inside \(\Spec(R(\beta))\).

In theory
nothing prevents us from restricting \(\mathcal{E}(L(\beta))\) to any subvariety
inside \(\Spec(R(\beta))\) but seems to us that only restriction to
the Lagrangian (or isotropic)  subvariety descends to the level of braids, not just closures
of the braids. That motivates us to put forward somewhat speculative

\begin{conjecture}
  There is a complex of filtred \(D(\beta)\)-modules \(\tilde{\mathcal{E}}(L(\beta))\)
  such that the associated graded module recovers our invariant:
  \[\mathrm{gr}(\tilde{\mathcal{E}}(L(\beta)))=\mathcal{E}(L(\beta)).\]
\end{conjecture}

In the case of torus links \(T_{n,k}\) the triply-graded homology have another interpretation in terms of
homology the homogeneous Hitchin fiber \cite{OblomkovYun16}. The homology of the homogeneous Hitchin fiber is  a fiber of
the complex of constructible sheaves on the Hitchin base. We expect that the complex of the D-modules \(\tilde{\mathcal{E}}(T_{n,k})\)
is the Riemann-Hilbert transform of the mentioned complex of constructible sheaves on the torus-fixed part of Hitchin base.


%
\end{document}